\documentclass[11pt]{amsart}

\usepackage[T1]{fontenc}
\usepackage[utf8]{inputenc}
\usepackage{lmodern}
\usepackage{amsmath,amssymb,amsthm,mathtools}
\usepackage{microtype}
\usepackage{geometry}
\usepackage{enumitem}
\usepackage{booktabs,tabularx}
\usepackage{placeins}
\usepackage[colorlinks=true,linkcolor=blue,citecolor=blue,urlcolor=blue]{hyperref}

\hypersetup{
  pdftitle={Nonnegative Bakry--Emery Curvature on Bounded-Degree Graphs Implies Volume Doubling and Poincare Inequalities},
  pdfauthor={Qi Guo; Xueping Huang; Yi C. Huang},
  pdfsubject={Bakry--Emery curvature on graphs, volume doubling, and Poincare inequalities},
  pdfkeywords={Bakry--Emery curvature, graphs, volume doubling, Poincare inequality, Point-mass estimates, resolvent smoothing}
}
\allowdisplaybreaks
\setlist{nosep}

\newtheorem{theorem}{Theorem}[section]
\newtheorem{lemma}[theorem]{Lemma}
\newtheorem{proposition}[theorem]{Proposition}
\newtheorem{corollary}[theorem]{Corollary}
\newtheorem{conjecture}{Conjecture}
\newtheorem{assumption}{Assumption}

\newtheoremstyle{italicremark}
  {6pt}{6pt}{\itshape}{}{\itshape}{.}{0.5em}{}
\theoremstyle{italicremark}
\newtheorem*{remark}{Remark}
\theoremstyle{plain}

\newcommand{\dist}{\operatorname{dist}}
\newcommand{\osc}{\operatorname{osc}}
\newcommand{\Var}{\operatorname{Var}}
\newcommand{\supp}{\operatorname{supp}}
\newcommand{\TV}{\mathrm{TV}}
\newcommand{\one}{\mathbf{1}}
\newcommand{\R}{\mathbb{R}}
\newcommand{\ip}[2]{\left\langle #1,#2\right\rangle}
\newcommand{\dd}{\,\mathrm{d}}
\newcommand{\e}{\mathrm{e}}
\newcommand{\norm}[1]{\left\lVert #1\right\rVert}
\newcommand{\abs}[1]{\left\lvert #1\right\rvert}
\newcommand{\pos}[1]{\left(#1\right)_+}

\title[Nonnegative Bakry--\'Emery Curvature Implies Doubling and Poincar\'e]
{Nonnegative Bakry--\'Emery Curvature on Bounded-Degree Graphs Implies
Volume Doubling and Poincar\'e Inequalities}

\author[Q. Guo]{Qi Guo}
\address{School of Mathematics, Renmin University of China, Beijing, 100872, P.R. China}
\email{qguo@ruc.edu.cn}

\author[X. Huang]{Xueping Huang}
\address{School of Mathematics and Statistics, Nanjing University of Information Science and Technology, Nanjing 210044, P.R. China}
\email{hxp@nuist.edu.cn}

\author[Y. C. Huang]{Yi C. Huang}
\address{School of Mathematical Sciences, Nanjing Normal University, Nanjing 210023, P.R. China}
\email{Yi.Huang.Analysis@gmail.com}

\subjclass[2020]{05C50, 60J27, 58J35}
\keywords{Bakry--\'Emery curvature, volume doubling, Poincar\'e inequality}
\date{}

\begin{document}

\begin{abstract}
We prove that every connected simple graph of bounded degree satisfying the classical
dimension-free Bakry--\'Emery condition $\mathrm{CD}(0,\infty)$ for the unnormalised
Laplacian is volume doubling and supports, at all integer graph scales, a scale-invariant
$L^2$-Poincar\'e inequality with dilation two, with constants depending only on the maximum
degree. This settles the polynomial-growth conjecture of Cushing, Liu, and Peyerimhoff in a
stronger form. The main novelty is a dimension-free adaptation of the graph-theoretic modified
nonlinear heat-flow method introduced by M\"unch and extended to infinite weighted graphs by
Pajot and Russ: point-mass consequences of $\Gamma_2\geq0$ and positive-resolvent smoothing
replace any global $\mathrm{CD}(0,n)$ reduction, while diffusive exit-time control and
finite-volume localisation yield the Poincar\'e inequality.
\end{abstract}

\maketitle

\setcounter{tocdepth}{1}
\tableofcontents

\section{Introduction}

On a complete Riemannian manifold, the landmark differential gradient estimate of Li and
Yau controls positive solutions of the heat equation in terms of a lower Ricci-curvature
bound \cite{LiYau}. In the case of nonnegative Ricci curvature, its integrated form gives a
parabolic Harnack inequality and, together with volume comparison, Gaussian heat-kernel
estimates. More generally, the work of Grigor'yan and Saloff-Coste shows that, for complete
Riemannian manifolds, the following three properties are quantitatively equivalent
\cite{GrigoryanHeat,SaloffCosteHarnack}:
\begin{enumerate}[label=(\roman*)]
\item volume doubling together with a scale-invariant $L^2$-Poincar\'e inequality;
\item a scale-invariant parabolic Harnack inequality;
\item two-sided Gaussian heat-kernel estimates.
\end{enumerate}
This equivalence makes volume doubling and Poincar\'e inequalities central geometric
counterparts of heat-flow regularity.

There are two principal approaches to formulating lower Ricci bounds beyond the smooth
setting. The Bakry--\'Emery approach encodes curvature through the Bochner inequality and
$\Gamma_2$-calculus for a Markov semigroup \cite{BE,BGL}. The approach of Lott--Villani
and Sturm instead uses displacement convexity of entropy along Wasserstein geodesics
\cite{LottVillani,SturmI,SturmII}. On infinitesimally Hilbertian metric-measure spaces,
the corresponding Riemannian curvature-dimension and Bakry--\'Emery formulations are
equivalent \cite{AGS,EKS}.

Both approaches have discrete analogues, but they lead to several inequivalent notions. For
graph Laplacians, the Bakry--\'Emery formalism leads to classical curvature-dimension
inequalities \cite{LY,BHLLMY}, whereas transport ideas lead, among
other notions, to Ollivier's coarse Ricci curvature and the entropic Ricci curvature of
Erbar and Maas \cite{Ollivier,ErbarMaas}. There is no general identification among these
discrete curvatures. Moreover, a graph Laplacian is not a diffusion operator: the chain
rule used in the classical logarithmic Li--Yau argument fails. Within the Bakry--\'Emery
formalism, $\mathrm{CD}(0,\infty)$ is the direct discrete version of the nonnegative
Bochner inequality, yet it was not known whether, on bounded-degree simple graphs, this
condition alone forces volume doubling or a scale-invariant ball Poincar\'e inequality
with constants depending only on the degree bound.

In this paper we resolve both questions. Our argument builds on the modified nonlinear
heat equation introduced by M\"unch for finite graphs and extended by Pajot and Russ to
infinite bounded-geometry graphs \cite{MunchCDN,RussPajot}. We prove that classical
nonnegative Bakry--\'Emery curvature controls the large-scale geometry of every connected
simple graph of bounded degree: it implies uniform volume doubling, polynomial volume
growth, and a scale-invariant local $L^2$-Poincar\'e inequality. The curvature hypothesis is
the dimension-free quadratic condition $\mathrm{CD}(0,\infty)$ for the unnormalised
Laplacian; no finite-dimensional curvature condition, exponential curvature-dimension
condition, regularity, or edge-regularity is assumed. The new dimension-free input consists
of point-mass consequences of the positive semidefinite form $\Gamma_2(\cdot,\cdot)(x)$
and a positive-resolvent smoothing estimate. The Poincar\'e inequality additionally requires
a heat-kernel localisation argument.

Throughout the paper, $G=(V,E)$ denotes a connected simple undirected graph. We work with
the unnormalised graph Laplacian
\[
\Delta f(x)=\sum_{y\sim x}\bigl(f(y)-f(x)\bigr).
\]
Its carr\'e du champ and iterated carr\'e du champ are denoted simply by
\[
\Gamma(f,g)
 =\frac12\bigl(\Delta(fg)-f\Delta g-g\Delta f\bigr)
\]
and
\[
\Gamma_2(f,g)
 =\frac12\bigl(\Delta\Gamma(f,g)
 -\Gamma(f,\Delta g)-\Gamma(g,\Delta f)\bigr).
\]
We write $\Gamma(f)=\Gamma(f,f)$ and $\Gamma_2(f)=\Gamma_2(f,f)$. The following two
assumptions will be used throughout.

\begin{assumption}\label{ass:degree}
The graph has bounded degree:
\[
d_*:=\sup_{x\in V}\deg(x)<\infty.
\]
\end{assumption}

\begin{assumption}\label{ass:curvature}
The unnormalised Laplacian satisfies the classical condition $\mathrm{CD}(0,\infty)$:
\[
\Gamma_2(f)(x)\geq0
\]
for every $x\in V$ and every real-valued function $f$ on $V$.
\end{assumption}

Assumption~\ref{ass:degree} implies local finiteness. Moreover,
$\Gamma_2(f)(x)$ depends only on the values of $f$ in the two-ball about $x$, so
Assumption~\ref{ass:curvature} is meaningful on infinite graphs and is unchanged if one tests
only finitely supported functions. We use open graph balls
\[
B(x,r):=\{y\in V:\dist(x,y)<r\}.
\]
For integer radii this convention differs from the closed-ball convention used in parts of the
literature only by shifting the radius by one. All volume-doubling and ball-Poincar\'e statements
below are formulated at positive integer radii, the natural graph scales.

\subsection{Curvature conditions and discrete Li--Yau theories}

For graph Laplacians, the classical carr\'e-du-champ theory was developed by Lin and Yau
\cite{LY}, Bauer, Horn, Lin, Lippner, Mangoubi, and Yau \cite{BHLLMY}, and many subsequent
authors, building on the diffusion formalism of Bakry and \`Emery \cite{BE,BGL}. The local
curvature-function and matrix viewpoint was developed systematically by Cushing, Liu, and
Peyerimhoff \cite{CLP}. These formulations make the local nature of graph Bakry--\'Emery
curvature explicit, but they do not restore the diffusion chain rule.

Several curvature conditions with similar notation must be kept separate. For $n<\infty$, the
classical condition is
\[
\mathrm{CD}(K,n):\qquad
\Gamma_2(f)\geq K\Gamma(f)+\frac1n(\Delta f)^2.
\]
Thus $\mathrm{CD}(K,n)$ implies $\mathrm{CD}(K,\infty)$, and the latter contains no finite
parameter. A graph Laplacian is not a diffusion operator, however, so the chain rule needed in
the manifold proof of the logarithmic Li--Yau estimate fails. This led to several nonlinear
substitutes.

Bauer et al. introduced the exponential curvature-dimension framework
$\mathrm{CDE}(n,K)$ and proved Li--Yau, Harnack, heat-kernel, Buser, and volume-growth
consequences \cite{BHLLMY}. Horn, Lin, Liu, and Yau later worked under the stronger
condition $\mathrm{CDE}'(n,K)$ \cite{HLLY}; M\"unch proved that
$\mathrm{CDE}'(n,K)$ implies the classical condition $\mathrm{CD}(K,n)$
\cite{MunchRemarks}. Hence $\mathrm{CDE}'(n,0)$ is formally stronger than the hypothesis of
the present paper both in curvature type and in dimension. By contrast,
$\mathrm{CDE}(n,0)$ is a distinct finite-dimensional nonlinear condition, and no general
implication from the classical condition used here to $\mathrm{CDE}(n,0)$ is known.

A different nonlinear approach was developed by M\"unch using a modified carr\'e-du-champ
calculus \cite{MunchNonlinear}. Dier, Kassmann, and Zacher introduced a
$\mathrm{CD}(F;0)$ condition and obtained discrete logarithmic Li--Yau and differential Harnack
estimates for positive solutions of the standard heat equation \cite{DKZ}. Their condition is
again distinct: its general relationship with the classical $\mathrm{CD}$ and exponential
$\mathrm{CDE}$ conditions is not known. These works are important predecessors for the
heat-flow viewpoint, but they do not by themselves give the volume-doubling and metric-ball
Poincar\'e statements proved below under classical $\mathrm{CD}(0,\infty)$.

For the classical condition itself, Lin and Liu identified the bounded-generator equivalence with
heat-semigroup gradient and variance inequalities \cite{LinLiu}; Keller and M\"unch obtained
related semigroup characterisations for unbounded graph Laplacians under suitable domain and
ellipticity hypotheses \cite{KellerMunch}. In particular, $\mathrm{CD}(0,\infty)$ yields the
pointwise semigroup estimate often called a ``reverse Poincar\'e inequality.'' This terminology
should not be confused with the local scale-invariant ball Poincar\'e inequality in
Theorem~\ref{thm:poincare-two} below. The same warning applies to pseudo-Poincar\'e estimates
and to global Poincar\'e, spectral-gap, or Buser inequalities
\cite{LiuBuser,KKRT,LakzianMcGuirk}: none of these is, without an additional localisation
argument, a uniform Neumann-type inequality on every metric ball.

\subsection{The volume-growth problem and its history}

The expectation that nonnegative discrete curvature should force polynomial volume growth can
be traced at least implicitly to Bauer et al.\ \cite{BHLLMY}. Under
$\mathrm{CDE}(n,0)$ they proved that bounded-degree graphs have polynomial volume growth
(in fact an upper bound of order $r^{2n}$ in their normalisation), together with Li--Yau and
standard-heat Harnack estimates and one-sided heat-kernel bounds. This was the first general
theorem of the form ``nonnegative curvature implies polynomial growth'' in the discrete
Bakry--\'Emery setting, and it is the natural historical source of the later folklore conjecture.
The general $\mathrm{CDE}(n,0)$ argument did not, however, yield scale-uniform volume doubling
or the metric-ball Poincar\'e inequality; the usual local package was recovered there under
additional strong cut-off hypotheses.

Horn, Lin, Liu, and Yau obtained a sharper analytic package under
$\mathrm{CDE}'(n,0)$ \cite{HLLY}. Their semigroup argument proves volume doubling under
$\mathrm{CDE}'(n,0)$; after the usual $\Delta(\alpha)$ ellipticity and laziness condition is
imposed, they obtain two-sided discrete-time Gaussian estimates, the weak scale-invariant
Poincar\'e inequality, and the parabolic Harnack inequality through Delmotte's theorem
\cite{Delmotte}. The logical distinction is important: by \cite{MunchRemarks}, their
hypothesis implies a finite-dimensional classical condition, whereas Assumption~\ref{ass:curvature}
is dimension-free.

The bounded-degree conjecture for the \emph{classical} condition was formulated explicitly by
Cushing, Liu, and Peyerimhoff as a discrete Bishop-type problem
\cite[Conjecture~9.10]{CLP}. They asked whether there are constants depending only on the
degree bound such that
\[
\#B(x,r)\leq C_1(d_*)\bigl(1+r^{C_2(d_*)}\bigr)
\]
for every graph satisfying Assumptions~\ref{ass:degree} and~\ref{ass:curvature}. Their work
also developed the local curvature-function and matrix viewpoint and exhibited finite-dimensional
obstructions. Salez later proved that bounded-degree expander families cannot satisfy
$\mathrm{CD}(0,\infty)$ \cite{Salez}, settling a necessary consequence of the conjecture but
not controlling volume growth or providing a scale-by-scale volume comparison.

Classical finite dimension was understood more recently. M\"unch introduced the modified heat
equation
\[
\partial_tu=\Delta u+\Gamma(u)
\]
and proved volume doubling for finite graphs satisfying $\mathrm{CD}(0,n)$
\cite{MunchCDN}. Pajot and Russ extended the construction and the doubling theorem to finite
or infinite weighted graphs of bounded geometry \cite{RussPajot}. Their Harnack inequality is
for this nonlinear modified heat equation; it is therefore not a standard heat-kernel Harnack
inequality and does not directly yield the metric-ball Poincar\'e inequality through Delmotte's
equivalence.

Blachar, Pajot, and Salez recently reformulated the classical dimension-free conjecture in the
integer-radius form stated below and proved the stronger volume-doubling conclusion for
edge-regular graphs \cite{BPS}. Their key input is an optimal finite-dimensional
self-improvement depending on the degree and the common-neighbour parameter. This includes
many symmetric examples but not arbitrary bounded-degree graphs, so the general bounded-degree
case remained open.

The bounded-geometry hypothesis cannot simply be omitted or replaced by an arbitrary
normalised-Laplacian statement. The exponentially growing normalised antitrees in
\cite{Antitrees} have nonnegative (indeed positive pointwise) normalised Bakry--\'Emery
curvature while failing polynomial growth and volume doubling; their degrees are unbounded.
This is why the physical, unnormalised Laplacian and the uniform degree bound in
Assumptions~\ref{ass:degree}--\ref{ass:curvature} are substantive.

With our open-ball convention, the recent integer-radius formulation is as follows.

\begin{conjecture}\label{conj:poly}
Suppose Assumptions~\ref{ass:degree} and~\ref{ass:curvature} hold. Then there is a finite
exponent $D(d_*)$ such that
\[
\#B(x,r)\leq r^{D(d_*)}
\]
for every $x\in V$ and every integer $r\geq1$.
\end{conjecture}

The formulation with a multiplicative constant in \cite{CLP} and the normalised
integer-radius formulation in \cite{BPS} are equivalent after changing the exponent by an
amount depending only on $d_*$. Conjecture~\ref{conj:poly} asks whether a local,
infinite-dimensional quadratic inequality imposes a uniform finite-dimensional geometry at all
scales. It asks for substantially more than the absence of expanders, and its difficulty is
specifically discrete: the standard manifold proof uses both a finite dimension and a diffusion
chain rule, neither of which is available under Assumption~\ref{ass:curvature}.

Our first main result proves the conjecture in the stronger doubling form.

\begin{theorem}\label{thm:doubling}
Suppose Assumptions~\ref{ass:degree} and~\ref{ass:curvature} hold. Then there is a finite
constant $K(d_*)$ such that
\[
\#B(x,2r)\leq K(d_*)\,\#B(x,r)
\]
for every $x\in V$ and every integer $r\geq1$.
\end{theorem}

Theorem~\ref{thm:doubling} is stronger at the conclusion level than the polynomial upper bound
proved under $\mathrm{CDE}(n,0)$ in \cite{BHLLMY} and than the conclusions conjectured in
\cite{CLP,BPS}: doubling controls the ratio of two consecutive scales and therefore supplies
both upper and lower interscale information. At the assumption level it removes the finite
classical dimension required in \cite{MunchCDN,RussPajot}, the
$\mathrm{CDE}'$ hypothesis used in \cite{HLLY}, and the edge-regularity used in
\cite{BPS}. Relative to \cite{BHLLMY}, the curvature conditions are different and no formal
implication is asserted; what is removed here is their finite-dimensional exponential
curvature hypothesis.

Iterating the estimate at dyadic radii gives the conjectured polynomial bound.

\begin{corollary}\label{cor:poly}
Suppose Assumptions~\ref{ass:degree} and~\ref{ass:curvature} hold. Then
\[
\#B(x,r)\leq r^{D(d_*)}
\]
for every $x\in V$ and every integer $r\geq1$, where one may take
\[
D(d_*):=\left\lceil2\log_2K(d_*)\right\rceil.
\]
\end{corollary}

\subsection{The local Poincar\'e problem}

For a nonempty finite set $D\subset V$, define
\[
f_D:=\frac1{\#D}\sum_{x\in D}f(x)
\]
and
\[
\mathcal E_D(f)
 :=\sum_{\substack{\{x,y\}\in E\\x,y\in D}}
 \bigl(f(x)-f(y)\bigr)^2.
\]
The weak scale-invariant ball Poincar\'e inequality asks for constants $C$ and a fixed dilation
$\lambda$ such that
\[
\sum_{z\in B(x,r)}|f(z)-f_{B(x,r)}|^2
\leq Cr^2\mathcal E_{B(x,\lambda r)}(f)
\]
uniformly in $x$, positive integer $r$, and $f$. The word ``weak'' refers only to the enlarged
ball on the right-hand side.

On uniformly elliptic lazy reversible graphs, Delmotte proved the quantitative equivalence
between volume doubling plus this ball Poincar\'e inequality, parabolic Harnack inequalities,
and two-sided discrete-time Gaussian heat-kernel estimates \cite{Delmotte}. In particular, the
ball Poincar\'e inequality is an independent localisation requirement; it is not a formal
consequence of volume doubling. The finite-dimensional modified-heat results
\cite{MunchCDN,RussPajot} provide the volume input but, prior to the present work, did not give
a general ball Poincar\'e theorem for classical $\mathrm{CD}(0,n)$. Nor do the pointwise
semigroup reverse Poincar\'e estimates of \cite{LinLiu,KellerMunch}, the pseudo-Poincar\'e
estimate in \cite{LiuBuser}, the global Buser inequalities in \cite{KKRT,LiuBuser}, or the
conical-curvature global inequality in \cite{LakzianMcGuirk} furnish the required localisation
on arbitrary balls.

The strongest previous general positive theorem is the result of Horn, Lin, Liu, and Yau:
$\mathrm{CDE}'(n,0)$ gives doubling, and together with $\Delta(\alpha)$ it gives the ball
Poincar\'e inequality, Gaussian estimates, and parabolic Harnack inequalities
\cite{HLLY}. The present theorem reaches the ball Poincar\'e conclusion from the classical
dimension-free curvature condition together with the standing maximum-degree bound.

\begin{theorem}\label{thm:poincare-two}
Suppose Assumptions~\ref{ass:degree} and~\ref{ass:curvature} hold. Then there is a finite
constant $C_P(d_*)$ such that
\[
\sum_{z\in B(x,r)}\left|f(z)-f_{B(x,r)}\right|^2
\leq C_P(d_*)\,r^2\,\mathcal E_{B(x,2r)}(f)
\]
for every $x\in V$, every integer $r\geq1$, and every function $f:V\to\R$.
\end{theorem}

The closest results and the precise assumption--conclusion comparison are summarised in
Table~\ref{tab:comparison}. The table separates curvature hypotheses that are sometimes
conflated in the literature.

\begin{table}[t]
\caption{Comparison with the closest volume and Poincar\'e results.}
\label{tab:comparison}
\small
\setlength{\tabcolsep}{4pt}
\renewcommand{\arraystretch}{1.16}
\begin{tabularx}{\textwidth}{@{}>{\raggedright\arraybackslash}p{0.20\textwidth}
>{\raggedright\arraybackslash}X >{\raggedright\arraybackslash}X@{}}
\toprule
Work & Assumptions & Conclusions relevant here \\
\midrule
Bauer et al.\ \cite{BHLLMY}
& Bounded geometry and the finite-dimensional nonlinear condition $\mathrm{CDE}(n,0)$
& Li--Yau/Harnack estimates and polynomial upper growth; no general scale-uniform doubling or ball Poincar\'e theorem without additional strong cut-offs. \\
Dier--Kassmann--Zacher \cite{DKZ}
& The nonlinear condition $\mathrm{CD}(F;0)$, tailored to a discrete logarithmic Li--Yau method
& Li--Yau and differential Harnack estimates for positive standard-heat solutions; no volume-doubling or metric-ball Poincar\'e theorem under classical $\mathrm{CD}(0,\infty)$. \\
Lin--Liu and Keller--M\"unch \cite{LinLiu,KellerMunch}
& Classical $\mathrm{CD}(0,\infty)$, in bounded-generator and suitable unbounded-generator settings
& Heat-semigroup gradient and variance inequalities, including pointwise reverse Poincar\'e estimates; no volume comparison or local ball Poincar\'e conclusion by themselves. \\
Horn--Lin--Liu--Yau \cite{HLLY}
& $\mathrm{CDE}'(n,0)$; for the full heat-kernel package also $\Delta(\alpha)$
& Volume doubling; with $\Delta(\alpha)$, two-sided Gaussian bounds, weak ball Poincar\'e, and parabolic Harnack. \\
M\"unch; Pajot and Russ \cite{MunchCDN,RussPajot}
& Classical $\mathrm{CD}(0,n)$ with $n<\infty$ and bounded geometry
& Volume doubling via the modified nonlinear heat equation; no general standard-kernel ball Poincar\'e theorem in those works. \\
Blachar--Pajot--Salez \cite{BPS}
& Classical $\mathrm{CD}(0,\infty)$ plus edge-regularity
& Finite-dimensional self-improvement and volume doubling; the general edge-regular ball Poincar\'e problem is not supplied by self-improvement alone. \\
Present paper
& Classical $\mathrm{CD}(0,\infty)$ on connected simple graphs with the unnormalised
counting-measure Laplacian, together with a maximum-degree bound
& Uniform volume doubling and a weak $L^2$ ball Poincar\'e inequality with the conventional dilation $2$. \\
\bottomrule
\end{tabularx}
\end{table}
\FloatBarrier

\begin{remark}
The case $d_*=0$ is trivial. If $d_*\geq1$, consider the lazy discrete-time kernel
\[
  \widetilde Q:=I+\frac{\Delta}{2d_*}.
\]
Every edge transition has probability $1/(2d_*)$, and every holding probability is at least
$1/2$. Its Dirichlet form is $(2d_*)^{-1}\mathcal E$, so
Theorems~\ref{thm:doubling} and~\ref{thm:poincare-two} provide Delmotte's volume-doubling and
weak Poincar\'e hypotheses with constants depending only on $d_*$. Delmotte's equivalence
\cite{Delmotte} therefore yields the corresponding two-sided discrete-time Gaussian heat-kernel
estimates and standard parabolic Harnack inequality. Thus, under the standing bounded-degree
framework, the conclusion-level package of \cite{HLLY} is recovered from the weaker classical
dimension-free curvature hypothesis.
\end{remark}

\subsection{Why the proof does not use finite-dimensional self-improvement}

A natural route to Corollary~\ref{cor:poly} would be a local finite-dimensional
self-improvement. At a vertex $x$, restrict functions to the finite two-ball and set
\[
Q_x(f,g):=\Gamma_2(f,g)(x),
\qquad
\ell_x(f):=\Delta f(x).
\]
Under Assumption~\ref{ass:curvature}, the form $Q_x$ is positive semidefinite. A finite local
dimension $n_x$ with
\[
Q_x(f,f)\geq\frac{\ell_x(f)^2}{n_x}
\]
exists precisely when
\[
\ker Q_x\subseteq\ker\ell_x;
\]
in that case the optimal value is
\[
n_x^{\mathrm{opt}}
 =\sup_{Q_x(f,f)>0}\frac{\ell_x(f)^2}{Q_x(f,f)}<\infty.
\]
The edge-regular calculation of Blachar, Pajot, and Salez proves this kernel inclusion with an
optimal dimension depending on the degree and common-neighbour parameter, and they
conjectured that every $d$-regular $\mathrm{CD}(0,\infty)$ graph satisfies
$\mathrm{CD}(0,d)$ \cite[Conjecture~2]{BPS}.\footnote{In a private preprint (available upon request), we disprove Conjecture~2 of \cite{BPS}: for each $d\geq5$ there is a
	finite connected $d$-regular $\mathrm{CD}(0,\infty)$ graph satisfying no
	$\mathrm{CD}(0,n)$ with $n<\infty$.} The available finite-dimensional
self-improvement therefore requires additional local structure and is not available under the
standing bounded-degree assumptions. Accordingly, the present proof does not proceed through
a global finite-dimensional reduction.

Instead, we extract only the two consequences of the positive semidefinite form $Q_x$ that are
needed by the modified-heat argument. Testing $Q_x$ against the point mass at $x$ gives an
effective dimension estimate at a negative local minimum of the Laplacian. Factoring the same
coercive expression through a positive resolvent gives an $L^\infty$ estimate for $\Delta P_tf$. The
first estimate enters a modified Li--Yau maximum principle; the second controls semigroup
displacement. Together with two exponential comparisons for the modified heat flow, they yield
a Harnack inequality and then volume doubling.

The additional difficulty in Theorem~\ref{thm:poincare-two} is localisation. The standard heat
kernel has global support, whereas the desired right-hand side contains only the energy in
$B(x,2r)$. The classical reverse Poincar\'e estimate gives overlap of nearby heat-kernel rows,
the resolvent estimate gives a diffusive first-moment displacement bound, and a strong-Markov
argument converts this into exit-time control. Coupling with the finite-state chain obtained by
rejecting boundary-crossing jumps, followed by a finite-dimensional kernel-localisation lemma,
then yields a scale-invariant Poincar\'e inequality with a fixed degree-dependent dilation.
Theorem~\ref{thm:doubling} is used only to reduce that dilation to two.

The paper is organised as follows. Section~\ref{sec:preliminary} records the heat
semigroup generated directly by $\Delta$ and the basic $\Gamma$-calculus. Section~\ref{sec:point-resolvent} proves the
point-mass and resolvent estimates. Section~\ref{sec:modified-heat} develops the modified heat
equation and derives the Li--Yau and Harnack inequalities. Section~\ref{sec:doubling-proof}
proves Theorem~\ref{thm:doubling} and Corollary~\ref{cor:poly}.
Section~\ref{sec:poincare-localisation} establishes heat-kernel localisation and the
fixed-dilation Poincar\'e inequality, and Section~\ref{sec:poincare-proof} proves
Theorem~\ref{thm:poincare-two}. Appendix~\ref{sec:max-principle} contains the maximum
principle used on infinite graphs.

\section{Preliminary}\label{sec:preliminary}

Because the maximum degree is bounded, the unnormalised combinatorial Laplacian is already a
bounded operator on $\ell^\infty(V)$. We use $\Delta$ as the sole ambient generator in all
curvature and nonlinear-flow arguments, with heat semigroup $P_t=e^{t\Delta}$.
Uniformisation is used only to verify the Markov properties of $P_t$. A finite-volume censored
restriction $\Delta_D$ will be introduced in Section~\ref{sec:poincare-localisation} solely for
the localisation argument. We also record the $\ell^\infty$-norm continuity and differentiability
needed later, and write $d_x:=\deg(x)$.

The cases $d_*\le1$ are immediate: a connected simple graph of maximum degree at most one has at most two vertices. They will be included in the final constant by setting $K(d_*)=2$ in those cases. Henceforth assume
\begin{equation}
  d_*\ge2,
  \qquad
  N:=d_*(d_*+3).
  \label{eq:N-definition}
\end{equation}
The symbol $N$ is an effective parameter for the estimates below; it is not a claim that the graph satisfies the global condition $\mathrm{CD}(0,N)$.

Because $d_*<\infty$, the operator $\Delta$ is bounded on $\ell^\infty(V)$ and
\begin{equation}
  \norm{\Delta}_{\ell^\infty\to\ell^\infty}\le2d_*.
  \label{eq:Delta-bound}
\end{equation}
Define the heat semigroup directly by
\begin{equation}
  P_t:=\e^{t\Delta},
  \qquad t\ge0.
  \label{eq:heat-semigroup}
\end{equation}
Its Markov properties follow, for example, from the uniformisation formula
\begin{equation}
  P_t
  =\e^{-d_*t}
   \sum_{n=0}^\infty
   \frac{(d_*t)^n}{n!}
   \left(I+\frac{\Delta}{d_*}\right)^n.
  \label{eq:poisson-expansion}
\end{equation}
Indeed, $I+d_*^{-1}\Delta$ is the symmetric stochastic operator with edge transition probabilities $d_*^{-1}$ and holding probability $1-d_x/d_*$ at $x$. Thus $P_t$ has a symmetric stochastic kernel $p_t(x,y)$, is a positive contraction on every $\ell^q(V)$, $1\le q\le\infty$, is self-adjoint on $\ell^2(V)$, and satisfies $P_t\one=\one$. The use of \eqref{eq:poisson-expansion} is only to verify these semigroup properties; all subsequent differential and curvature calculations are performed with the single Laplacian $\Delta$.

Because $G$ is connected and locally finite, $V$ is countable. Symmetry and stochasticity therefore imply column stochasticity. In particular, for every nonnegative $h\in\ell^1(V)$, Tonelli's theorem gives
\begin{equation}
  \sum_{x\in V}P_th(x)
  =\sum_{y\in V}h(y)\sum_{x\in V}p_t(x,y)
  =\sum_{y\in V}h(y).
  \label{eq:mass-conservation}
\end{equation}
All semigroup differentiations below are bounded-operator differentiations on $\ell^\infty(V)$.
More precisely, $t\mapsto P_t$ is real analytic in operator norm, and becomes entire after
complexification, with
\[
  \frac{\dd}{\dd t}P_t=\Delta P_t=P_t\Delta.
\]
The bilinear maps $\Gamma$ and $\Gamma_2$ are bounded:
\begin{align}
  \norm{\Gamma(f,g)}_\infty
  &\le 2d_*\norm{f}_\infty\norm{g}_\infty,
  \label{eq:Gamma-bilinear-bound}\\
  \norm{\Gamma_2(f,g)}_\infty
  &\le 6d_*^2\norm{f}_\infty\norm{g}_\infty.
  \label{eq:Gamma2-bilinear-bound}
\end{align}
Consequently their associated quadratic maps are $C^\infty$, with
\[
  D[\Gamma(\,\cdot\,)](u)h=2\Gamma(u,h),
  \qquad
  D[\Gamma_2(\,\cdot\,)](u)h=2\Gamma_2(u,h).
\]
For every $\gamma\in\R$, the pointwise exponential map
$E_\gamma(u):=\e^{\gamma u}$ is $C^\infty$ on $\ell^\infty(V)$ and
\[
  DE_\gamma(u)h=\gamma\e^{\gamma u}h.
\]
Thus all differentiations below are ordinary Fr\'echet differentiations of norm-$C^1$
Banach-valued curves. Every Banach-valued time integral is a Bochner integral in the space
indicated. In particular, the interpolation integrals are $\ell^\infty(V)$-valued, whereas the
resolvent integral is taken in operator norm in $\mathcal B(\ell^\infty(V))$.

The pointwise edge formula for the carr\'e du champ is
\begin{equation}
  \Gamma(f,g)(x)
  =\frac12\sum_{y\sim x}
    \bigl(f(y)-f(x)\bigr)\bigl(g(y)-g(x)\bigr).
  \label{eq:edge-gamma}
\end{equation}
We assume $\mathrm{CD}(0,\infty)$ from now on:
\begin{equation}
  \Gamma_2(f)(x)\ge0
  \qquad(f:V\to\R,\ x\in V).
  \label{eq:standing-CD}
\end{equation}
The first identity below is the defining relation for $\Gamma_2$, rewritten:
\begin{equation}
  \Delta\Gamma(f)=2\Gamma(f,\Delta f)+2\Gamma_2(f).
  \label{eq:bochner}
\end{equation}
The second is the standard semigroup interpolation identity. See \cite{BGL} for the general
$\Gamma$-calculus framework and \cite{LinLiu} for graph-specific semigroup formulations. For
bounded $f$ and $t>0$,
\begin{equation}
  P_t\Gamma(f)-\Gamma(P_tf)
  =2\int_0^t P_s\Gamma_2(P_{t-s}f)\,\dd s.
  \label{eq:interpolation}
\end{equation}
Indeed, differentiating in $\ell^\infty(V)$ the function
\[
  \Phi(s):=P_s\Gamma(P_{t-s}f)
\]
gives $\Phi'(s)=2P_s\Gamma_2(P_{t-s}f)$. The integrand in \eqref{eq:interpolation} is norm-continuous: $s\mapsto P_{t-s}f$ is norm-continuous and $\Gamma_2$ is a continuous quadratic map from $\ell^\infty(V)$ to itself because $\Delta$ is bounded. Thus the integral is an ordinary Bochner integral in $\ell^\infty(V)$.

\section{Point-mass estimates and resolvent smoothing}\label{sec:point-resolvent}

This section extracts the two finite-dimensional substitutes used in the volume and localisation
arguments. First, a point-mass test of the positive semidefinite form
$\Gamma_2(\cdot,\cdot)(x)$ gives a second-order coercive inequality and an effective dimension
estimate at a negative local minimum of $\Delta f$. Second, a positive resolvent turns the same
coercive expression into a global $L^\infty$ estimate for $\Delta P_t f$. The first estimate is
used in the Li--Yau maximum principle, and the second controls semigroup displacement and
exit-time localisation.

\subsection{Point-mass estimates}\label{subsec:point-mass}

For $x\in V$, write $e_x:=\one_{\{x\}}$. All quantities in the next lemma depend only on the restriction of $f$ to the two-ball about $x$, so no boundedness assumption on $f$ is needed. The identity is the algebraic core of the proof.

\begin{lemma}[Point-mass identities]
\label{lem:point-mass}
For every $x\in V$ and every function $f:V\to\R$,
\begin{align}
  4\Gamma_2(f,e_x)(x)
  &=\Delta^2f(x)-2\Delta f(x),
  \label{eq:point-mass-bilinear}\\
  4\Gamma_2(e_x)(x)
  &=d_x(d_x+3).
  \label{eq:point-mass-quadratic}
\end{align}
\end{lemma}

\begin{proof}
At the vertex $x$,
\[
  \Delta e_x(x)=-d_x,
  \qquad
  \Delta e_x(y)=1\quad(y\sim x),
  \qquad
  \Gamma(f,e_x)(x)=-\frac12\Delta f(x).
\]
For $y\sim x$, only the edge from $y$ to $x$ contributes to the $e_x$-difference, and hence
\[
  \Gamma(f,e_x)(y)=\frac12\bigl(f(x)-f(y)\bigr).
\]
It follows term by term that
\begin{align*}
  \Delta\Gamma(f,e_x)(x)
  &=\sum_{y\sim x}
    \left[
      \frac12\bigl(f(x)-f(y)\bigr)
      +\frac12\Delta f(x)
    \right]
    =\frac{d_x-1}{2}\Delta f(x),\\
  \Gamma(f,\Delta e_x)(x)
  &=\frac12\sum_{y\sim x}
    \bigl(f(y)-f(x)\bigr)(1+d_x)
    =\frac{d_x+1}{2}\Delta f(x),\\
  \Gamma(e_x,\Delta f)(x)
  &=-\frac12\sum_{y\sim x}
    \bigl(\Delta f(y)-\Delta f(x)\bigr)
    =-\frac12\Delta^2f(x).
\end{align*}
Substitution into the bilinear definition of $\Gamma_2$ yields \eqref{eq:point-mass-bilinear}. Taking $f=e_x$ and computing
\[
  \Delta^2e_x(x)
  =\sum_{y\sim x}\bigl(\Delta e_x(y)-\Delta e_x(x)\bigr)
  =\sum_{y\sim x}(1+d_x)
  =d_x(d_x+1)
\]
gives \eqref{eq:point-mass-quadratic}.
\end{proof}

At a fixed vertex $x$, the map
\[
  Q_x(f,g):=\Gamma_2(f,g)(x)
\]
is a symmetric positive semidefinite bilinear form by \eqref{eq:standing-CD}. Hence it obeys Cauchy--Schwarz: applying $Q_x(f+tg,f+tg)\ge0$ for every $t\in\R$ and inspecting the discriminant gives
\[
  Q_x(f,g)^2\le Q_x(f,f)Q_x(g,g).
\]
Combining this observation with Lemma~\ref{lem:point-mass} gives the following coercivity estimate.

\begin{proposition}[Second-order coercivity]
\label{prop:second-order-coercivity}
For every $x\in V$ and every $f:V\to\R$,
\begin{equation}
  \bigl(\Delta^2f(x)-2\Delta f(x)\bigr)^2
  \le4d_x(d_x+3)\Gamma_2(f)(x).
  \label{eq:local-coercivity}
\end{equation}
In particular,
\begin{equation}
  \bigl(\Delta^2f-2\Delta f\bigr)^2
  \le4N\Gamma_2(f)
  \label{eq:global-coercivity}
\end{equation}
pointwise on $V$.
\end{proposition}

\begin{proof}
Cauchy--Schwarz for $Q_x$ and Lemma~\ref{lem:point-mass} give
\[
  \frac1{16}\bigl(\Delta^2f(x)-2\Delta f(x)\bigr)^2
  =Q_x(f,e_x)^2
  \le Q_x(f,f)Q_x(e_x,e_x)
  =\frac{d_x(d_x+3)}4\Gamma_2(f)(x).
\]
This proves \eqref{eq:local-coercivity}; \eqref{eq:global-coercivity} follows from $d_x(d_x+3)\le N$.
\end{proof}

The first targeted finite-dimensional consequence occurs at extrema of $\Delta f$.

\begin{proposition}[Extremal curvature--dimension inequality]
\label{prop:extremal-CD}
Let $x\in V$ and let $f:V\to\R$ satisfy
\begin{equation}
  \Delta f(x)<0,
  \qquad
  \Delta f(y)\ge\Delta f(x)
  \quad\text{for every }y\sim x.
  \label{eq:exact-extremum}
\end{equation}
Then
\begin{equation}
  \Gamma_2(f)(x)
  \ge\frac{(\Delta f(x))^2}{d_x(d_x+3)}
  \ge\frac{(\Delta f(x))^2}{N}.
  \label{eq:exact-extremal-CD}
\end{equation}
More generally, if $\eta\ge0$ and
\begin{equation}
  \Delta f(y)-\Delta f(x)\ge-\eta
  \qquad(y\sim x),
  \label{eq:approx-extremum}
\end{equation}
then
\begin{equation}
  \Gamma_2(f)(x)
  \ge
  \frac{\pos{-\Delta f(x)-\frac12d_x\eta}^{2}}
       {d_x(d_x+3)}.
  \label{eq:approx-extremal-CD}
\end{equation}
\end{proposition}

\begin{proof}
Under the standing assumption $d_*\ge2$, connectedness gives $d_x\ge1$. From \eqref{eq:approx-extremum},
\[
  \Delta^2f(x)
  =\sum_{y\sim x}\bigl(\Delta f(y)-\Delta f(x)\bigr)
  \ge-d_x\eta.
\]
Therefore
\[
  \Delta^2f(x)-2\Delta f(x)
  \ge-d_x\eta-2\Delta f(x)
  =2\left(-\Delta f(x)-\frac12d_x\eta\right).
\]
If the right-hand side is positive, square this inequality and use \eqref{eq:local-coercivity}; if it is nonpositive, \eqref{eq:approx-extremal-CD} is trivial. This proves the approximate estimate. Taking $\eta=0$ gives the first inequality in \eqref{eq:exact-extremal-CD}, and the second follows from $d_x(d_x+3)\le N$.
\end{proof}

\begin{remark}[Why this does not imply global finite dimension]
Suppose $\Gamma_2(f)(x)=0$. Positivity of $\Gamma_2(\cdot,\cdot)(x)$ then forces $\Gamma_2(f,e_x)(x)=0$, and Lemma~\ref{lem:point-mass} gives
\[
  \Delta^2f(x)=2\Delta f(x).
\]
If $\Delta f(x)<0$, this identity is incompatible with $x$ being a local minimum of $\Delta f$, because local minimality would imply $\Delta^2f(x)\ge0$. Thus every zero-curvature direction with negative Laplacian lies outside the extremal regime used in Proposition~\ref{prop:extremal-CD}. This is precisely how the argument remains compatible with the irregular counterexample to global finite-dimensional self-improvement in \cite{BPS}.
\end{remark}

\subsection{Resolvent smoothing}\label{subsec:resolvent}

The coercive expression above factors as $A(A+2I)$ with $A=-\Delta$. We invert the positive
factor $A+2I$ and average the local $\Gamma_2$ control through positive heat kernels. This
removes the derivative that is not pointwise controllable and produces the semigroup estimate
used later.

The second use of Proposition~\ref{prop:second-order-coercivity} provides the semigroup estimate
that normally comes from a finite dimension parameter. Set
\begin{equation}
  A:=-\Delta,
  \qquad
  \mathcal T:=A(A+2I)=\Delta^2-2\Delta.
  \label{eq:resolvent-definitions}
\end{equation}
Define the bounded operator
\begin{equation}
  \mathcal R:=\int_0^\infty \e^{-2s}P_s\,\dd s,
  \label{eq:resolvent-representation}
\end{equation}
where the Bochner integral converges in operator norm in
$\mathcal B(\ell^\infty(V))$.

\begin{lemma}[Positive resolvent]
\label{lem:positive-resolvent}
The operator $\mathcal R$ satisfies
\[
  \mathcal R=(A+2I)^{-1}.
\]
It is a positive kernel operator and
\begin{equation}
  \mathcal R\one=\frac12\one.
  \label{eq:resolvent-mass}
\end{equation}
\end{lemma}

\begin{proof}
The convergence assertion follows from $\norm{P_s}_{\infty\to\infty}\le1$. Moreover,
\[
  \frac{\dd}{\dd s}\bigl(\e^{-2s}P_s\bigr)
  =-(A+2I)\e^{-2s}P_s.
\]
Integration from $0$ to $M$ and passage to the limit $M\to\infty$ show that $(A+2I)\mathcal R=I$. All operators involved are functions of the bounded operator $A$, so also $\mathcal R(A+2I)=I$.

Positivity follows from positivity of $P_s$. More explicitly, if $p_s(x,y)$ is the heat kernel, then
\begin{equation}
  r(x,y):=\int_0^\infty \e^{-2s}p_s(x,y)\,\dd s
  \label{eq:resolvent-kernel}
\end{equation}
is nonnegative and satisfies
\[
  \mathcal R h(x)=\sum_{y\in V}r(x,y)h(y)
\]
for bounded $h$; Tonelli's theorem, applied first to $\abs{h}$, justifies the interchange. Finally, $P_s\one=\one$ gives \eqref{eq:resolvent-mass}, equivalently
\[
  \sum_y r(x,y)=\frac12.
\]
\end{proof}

\begin{lemma}[Kernel Cauchy--Schwarz]
\label{lem:kernel-cs}
Let $\mathsf K$ be a positive kernel operator on $V$ with finite row masses. Then, for every
bounded real-valued $h$,
\begin{equation}
  (\mathsf K h)^2\le(\mathsf K\one)\mathsf K(h^2)
  \label{eq:kernel-CS}
\end{equation}
pointwise on $V$.
\end{lemma}

\begin{proof}
Write $\mathsf K h(x)=\sum_yk(x,y)h(y)$ with $k(x,y)\ge0$. At fixed $x$,
Cauchy--Schwarz in the finite measure $\sum_yk(x,y)\delta_y$ gives
\[
  \abs{\mathsf K h(x)}^2
  \le\left(\sum_yk(x,y)\right)
      \left(\sum_yk(x,y)h(y)^2\right),
\]
which is \eqref{eq:kernel-CS}.
\end{proof}

\begin{theorem}[Laplacian smoothing]
\label{thm:resolvent-smoothing}
For every bounded $f:V\to\R$, every $t>0$, and every $x\in V$,
\begin{equation}
  \abs{\Delta P_tf(x)}^2
  \le\frac{N}{t}
  \mathcal R\bigl(P_t\Gamma(f)-\Gamma(P_tf)\bigr)(x).
  \label{eq:smoothing-resolvent}
\end{equation}
In particular,
\begin{equation}
  \norm{\Delta P_tf}_\infty^2
  \le\frac{N}{2t}\norm{\Gamma(f)}_\infty.
  \label{eq:smoothing-linf}
\end{equation}
Consequently,
\begin{equation}
  \norm{P_tf-f}_\infty
  \le\sqrt{2Nt\,\norm{\Gamma(f)}_\infty}.
  \label{eq:semigroup-shift}
\end{equation}
\end{theorem}

\begin{proof}
Fix $s\in[0,t]$ and set $h_s:=P_{t-s}f$. Since $P_s$ and $\mathcal R$ are bounded
functions of $A$, all three relevant operators commute, and
\begin{equation}
  AP_tf=P_sAh_s=\mathcal R P_s\mathcal T h_s.
  \label{eq:factorisation-use}
\end{equation}
Here $P_sAh_s=AP_tf$ follows from the semigroup property, while
$\mathcal R\mathcal T=\mathcal R A(A+2I)=A$ because the factors commute.

For clarity, the positive kernel of $\mathcal R P_s$ is
\[
  k_s(x,y)
  =\int_0^\infty \e^{-2u}p_{u+s}(x,y)\,\dd u,
\]
and its row mass is exactly $1/2$. Apply Lemma~\ref{lem:kernel-cs} to
$\mathsf K=\mathcal R P_s$. By Proposition~\ref{prop:second-order-coercivity},
$(\mathcal T h_s)^2\le4N\Gamma_2(h_s)$. Hence, pointwise,
\begin{equation}
  (AP_tf)^2
  \le\frac12\mathcal R P_s\bigl((\mathcal T h_s)^2\bigr)
  \le2N\,\mathcal R P_s\Gamma_2(P_{t-s}f).
  \label{eq:resolvent-intermediate}
\end{equation}
The left-hand side does not depend on $s$. The map
$s\mapsto P_s\Gamma_2(P_{t-s}f)$ is norm-continuous in $\ell^\infty(V)$, and
$\mathcal R$ is bounded, so integrating \eqref{eq:resolvent-intermediate} as a Bochner
integral and using \eqref{eq:interpolation} gives
\[
  t(AP_tf)^2
  \le2N\mathcal R\int_0^tP_s\Gamma_2(P_{t-s}f)\,\dd s
  =N\mathcal R\bigl(P_t\Gamma(f)-\Gamma(P_tf)\bigr).
\]
This is \eqref{eq:smoothing-resolvent}, because $A=-\Delta$.

Under $\mathrm{CD}(0,\infty)$ the bracket in \eqref{eq:smoothing-resolvent} is
nonnegative by \eqref{eq:interpolation}. Since $\Gamma(P_tf)\ge0$ and $\mathcal R$ is
positive,
\[
  0\le\mathcal R\bigl(P_t\Gamma(f)-\Gamma(P_tf)\bigr)
  \le\mathcal R P_t\Gamma(f)
  \le\frac12\norm{\Gamma(f)}_\infty,
\]
where the last inequality uses positivity, $P_t\one=\one$, and
\eqref{eq:resolvent-mass}. This proves \eqref{eq:smoothing-linf}. Finally,
\[
  P_tf-f=\int_0^t\Delta P_sf\,\dd s
\]
in $\ell^\infty(V)$, and therefore
\[
  \norm{P_tf-f}_\infty
  \le\int_0^t
     \sqrt{\frac{N}{2s}\norm{\Gamma(f)}_\infty}\,\dd s
  =\sqrt{2Nt\,\norm{\Gamma(f)}_\infty}.
\]
\end{proof}

\begin{remark}
Neither \eqref{eq:exact-extremal-CD} nor \eqref{eq:smoothing-linf} implies a global $\mathrm{CD}(0,N)$ inequality. They are instead two projections of the same positive-semidefinite structure onto the point mass $e_x$. The first is adapted to a maximum principle; the second is adapted to semigroup interpolation through the factorisation $\mathcal T=A(A+2I)$.
\end{remark}

\section{Modified heat, Li--Yau, and Harnack estimates}\label{sec:modified-heat}

This section develops the analytic route from the point-mass estimates to Harnack control.
The modified heat equation replaces the unavailable diffusion chain rule; its gradient bound
and exponential comparisons are then combined with the extremal curvature estimate from
Subsection~\ref{subsec:point-mass}. The output is a Li--Yau inequality and a modified-flow Harnack
estimate, which are the final analytic inputs for volume doubling.

\subsection{The modified heat equation}\label{subsec:modified-flow}

M\"unch introduced, for finite graphs, the use of
\[
  \partial_tu=\Delta u+\Gamma(u)
\]
as a discrete substitute for $\log(P_t\e^{u_0})$, together with the gradient estimate and
pointwise comparisons with the linear heat semigroup
\cite[Theorems~2.1 and~2.3]{MunchCDN}. Pajot and Russ developed the corresponding
$L^\infty$ theory on possibly infinite weighted graphs of bounded geometry. Using linear
semigroup estimates and a convergent mild iteration, they proved global well-posedness, treated
the failure of spatial extrema to be attained, and extended M\"unch's comparison argument
\cite[Theorem~1.3 and Sections~5--6]{RussPajot}. In the present bounded-degree unnormalised
setting, the vector field $u\mapsto\Delta u+\Gamma(u)$ is locally Lipschitz on
$\ell^\infty(V)$, so we replace their mild-iteration construction by a direct Banach-space ODE
argument.

For bounded initial data $u_0$, consider
\begin{equation}
  \partial_tu_t=\Delta u_t+\Gamma(u_t),
  \qquad
  u_{t=0}=u_0.
  \label{eq:modified-heat}
\end{equation}
The next proposition records global existence and the gradient maximum principle in the present bounded-degree setting.

\begin{proposition}[Global nonlinear flow and gradient bound]
\label{prop:global-flow}
Let $u_0\in\ell^\infty(V)$ and assume
\begin{equation}
  \norm{\Gamma(u_0)}_\infty<\frac12.
  \label{eq:initial-gradient-smallness}
\end{equation}
Then \eqref{eq:modified-heat} has a unique global solution
\[
  u\in C^1\bigl([0,\infty),\ell^\infty(V)\bigr).
\]
Moreover,
\begin{equation}
  \norm{\Gamma(u_t)}_\infty
  \le\norm{\Gamma(u_0)}_\infty
  \qquad(t\ge0),
  \label{eq:gradient-max}
\end{equation}
and hence
\begin{equation}
  \abs{u_t(y)-u_t(x)}<1
  \qquad(x\sim y,\ t\ge0).
  \label{eq:edge-smallness}
\end{equation}
\end{proposition}

\begin{proof}
Set
\[
  \mathcal F(u):=\Delta u+\Gamma(u).
\]
This is a $C^\infty$ quadratic polynomial on $\ell^\infty(V)$, with
\[
  D\mathcal F(u)h=\Delta h+2\Gamma(u,h).
\]
In particular, $\mathcal F$ is locally Lipschitz. Indeed, if $\norm{u}_\infty,\norm{v}_\infty\le M$, then for every $x$,
\begin{align*}
  \abs{\Gamma(u)(x)-\Gamma(v)(x)}
  &\le\frac12\sum_{y\sim x}
   \abs{(u(y)-u(x))^2-(v(y)-v(x))^2}\\
  &\le4d_*M\norm{u-v}_\infty.
\end{align*}
Together with \eqref{eq:Delta-bound}, this gives
\[
  \norm{\mathcal F(u)-\mathcal F(v)}_\infty
  \le\bigl(2d_*+4d_*M\bigr)\norm{u-v}_\infty
\]
on the closed radius-$M$ ball. The Banach-space Picard--Lindel\"of theorem
\cite{DeimlingODE} therefore gives a unique maximal solution
\[
  u\in C^1\bigl([0,T_{\max}),\ell^\infty(V)\bigr).
\]
We postpone the continuation argument until after deriving the a priori bounds below.

Let
\[
  H_t:=\Gamma(u_t).
\]
The bilinearity and continuity of $\Gamma$ show that $H\in C^1([0,T_{\max}),\ell^\infty(V))$. Differentiating and using \eqref{eq:bochner} gives
\begin{align}
  \partial_tH_t
  &=2\Gamma(u_t,\Delta u_t)+2\Gamma(u_t,H_t)\notag\\
  &=\Delta H_t+2\Gamma(u_t,H_t)-2\Gamma_2(u_t).
  \label{eq:H-evolution}
\end{align}
Suppose for the moment that $\norm{H_t}_\infty<1/2$. For every edge $x\sim y$, the single-edge contribution to \eqref{eq:edge-gamma} gives
\[
  \frac12\abs{u_t(y)-u_t(x)}^2
  \le H_t(x)<\frac12,
\]
so $\abs{u_t(y)-u_t(x)}<1$. Since $\Gamma_2(u_t)\ge0$, equation \eqref{eq:H-evolution} becomes
\begin{equation}
  \partial_tH_t(x)
  \le\sum_{y\sim x}
    \bigl(1+u_t(y)-u_t(x)\bigr)
    \bigl(H_t(y)-H_t(x)\bigr).
  \label{eq:H-max-ineq}
\end{equation}
For $y\sim x$, put
\[
  a_t(x,y):=1+u_t(y)-u_t(x),
\]
and set $a_t(x,y):=0$ otherwise. These coefficients are nonnegative, and
\[
  \sum_y a_t(x,y)\le2d_x\le2d_*.
\]
All sums are absolutely convergent because each row has finite support. Lemma~\ref{lem:max-principle} implies that $t\mapsto\sup_xH_t(x)=\norm{H_t}_\infty$ is nonincreasing throughout every time interval on which $\norm{H_t}_\infty<1/2$.

We now close the bootstrap. Let $M_\Gamma:=\norm{H_0}_\infty<1/2$ and suppose that a first time $\tau<T_{\max}$ exists with $\norm{H_\tau}_\infty=1/2$. On $[0,\tau)$ the preceding maximum principle gives $\norm{H_t}_\infty\le M_\Gamma$. Norm-continuity of $H_t$ then implies $\norm{H_\tau}_\infty\le M_\Gamma<1/2$, a contradiction. Therefore \eqref{eq:gradient-max} holds on $[0,T_{\max})$, and \eqref{eq:edge-smallness} follows from the single-edge bound above.

Finally,
\[
  \norm{\partial_tu_t}_\infty
  \le2d_*\norm{u_t}_\infty+M_\Gamma.
\]
Hence
\[
  \norm{u_t}_\infty
  \le\norm{u_0}_\infty
     +\int_0^t\bigl(2d_*\norm{u_s}_\infty+M_\Gamma\bigr)\,\dd s.
\]
Gronwall's inequality gives
\begin{equation}
  \norm{u_t}_\infty
  \le\e^{2d_*t}\norm{u_0}_\infty
  +\frac{M_\Gamma}{2d_*}\bigl(\e^{2d_*t}-1\bigr).
  \label{eq:global-flow-bound}
\end{equation}
Suppose, for contradiction, that $T_{\max}<\infty$. By
\eqref{eq:global-flow-bound},
\[
  M:=\sup_{0\leq t<T_{\max}}\norm{u_t}_\infty<\infty.
\]
Together with \eqref{eq:gradient-max}, this gives
\[
  \norm{\partial_tu_t}_\infty
  =\norm{\Delta u_t+\Gamma(u_t)}_\infty
  \leq2d_*M+M_\Gamma
  \qquad(0\leq t<T_{\max}).
\]
Thus $u$ is uniformly Lipschitz on $[0,T_{\max})$ and consequently has a limit
$u_*\in\ell^\infty(V)$ as $t\uparrow T_{\max}$. Applying the Banach-space
Picard--Lindel\"of theorem at time $T_{\max}$ with initial value $u_*$ produces a local
solution beyond $T_{\max}$. Continuity of $\mathcal F$ and uniqueness glue it to the existing
solution, contradicting maximality. Therefore $T_{\max}=\infty$.
\end{proof}

Two exponential comparisons replace the formal chain rule that, in a diffusion setting, would identify $u_t$ with $\log(P_t\e^{u_0})$.

\begin{lemma}[Scalar inequalities]
\label{lem:scalar-inequalities}
For every $z\in[-1,1]$,
\begin{align}
  \frac12z\left(1+\frac z2\right)
  -\bigl(\e^{z/2}-1\bigr)&\ge0,
  \label{eq:scalar-small}\\
  2z\left(1+\frac z2\right)
  -\bigl(\e^{2z}-1\bigr)&\le0.
  \label{eq:scalar-large}
\end{align}
\end{lemma}

\begin{proof}
For \eqref{eq:scalar-small}, put $w=z/2\in[-1/2,1/2]$ and define
\[
  q(w):=1+w+w^2-\e^w.
\]
Since $q(0)=q'(0)=0$ and $q''(w)=2-\e^w>0$ on this interval, $w=0$ is the unique minimum of $q$.

For \eqref{eq:scalar-large}, first let $z\ge0$. Taylor's inequality $\e^{2z}\ge1+2z+2z^2$ gives
\[
  2z\left(1+\frac z2\right)-(\e^{2z}-1)
  \le-z^2\le0.
\]
Now write $z=-s$ with $s\in[0,1]$. The desired inequality is equivalent to
\begin{equation}
  \varphi(s):=2s-s^2-1+\e^{-2s}\ge0.
  \label{eq:phi}
\end{equation}
Here
\[
  \varphi'(s)=2g(s),
  \qquad
  g(s):=1-s-\e^{-2s}.
\]
The function $g$ is strictly concave, satisfies $g(0)=0$, $g'(0)=1$, and $g(1)=-\e^{-2}<0$. Its nonnegative superlevel set is therefore an interval; since $g$ is positive for small positive $s$ and negative at $1$, there is a unique sign change in $(0,1)$. Thus $\varphi$ first increases and then decreases. Its minimum on $[0,1]$ is attained at an endpoint, and $\varphi(0)=0$, $\varphi(1)=\e^{-2}>0$. This proves \eqref{eq:phi}.
\end{proof}

\begin{proposition}[Semigroup comparisons]
\label{prop:semigroup-comparisons}
Let $u_t$ be the solution in Proposition~\ref{prop:global-flow}. With
\begin{equation}
  \theta:=\frac12,
  \qquad
  \beta:=2,
  \label{eq:theta-beta}
\end{equation}
one has, pointwise on $V$,
\begin{align}
  \e^{\theta u_t}&\ge P_t\e^{\theta u_0},
  \label{eq:lower-comparison}\\
  \e^{\beta u_t}&\le P_t\e^{\beta u_0}.
  \label{eq:upper-comparison}
\end{align}
\end{proposition}

\begin{proof}
For $\gamma>0$ and fixed $t$, set
\[
  J_\gamma(s):=P_{t-s}\e^{\gamma u_s},
  \qquad 0\le s\le t.
\]
The Nemytskii map $v\mapsto \e^{\gamma v}$ is continuously Fr\'echet differentiable on
$\ell^\infty(V)$, with derivative $h\mapsto\gamma\e^{\gamma v}h$; on bounded subsets its
derivative is bounded and locally Lipschitz. Since $u\in C^1([0,\infty),\ell^\infty)$ and
$t\mapsto P_t$ is differentiable in operator norm, $J_\gamma$ is $C^1$ as an
$\ell^\infty(V)$-valued map, and differentiating gives
\begin{equation}
  J_\gamma'(s)
  =P_{t-s}\left[
     \gamma\e^{\gamma u_s}\bigl(\Delta u_s+\Gamma(u_s)\bigr)
     -\Delta\e^{\gamma u_s}
   \right].
  \label{eq:Jgamma-derivative}
\end{equation}
At a vertex $x$, writing $z_y:=u_s(y)-u_s(x)$, the expression inside the
brackets becomes
\begin{align}
 &\gamma\e^{\gamma u_s(x)}
   \bigl(\Delta u_s(x)+\Gamma(u_s)(x)\bigr)
   -\Delta\e^{\gamma u_s}(x)\notag\\
 &\qquad
  =\e^{\gamma u_s(x)}
   \sum_{y\sim x}
   \left[
     \gamma z_y\left(1+\frac{z_y}{2}\right)
     -\bigl(\e^{\gamma z_y}-1\bigr)
   \right].
  \label{eq:Jgamma-pointwise}
\end{align}
By \eqref{eq:edge-smallness}, every $z_y$ lies in $[-1,1]$. Lemma~\ref{lem:scalar-inequalities} shows that $J_\theta$ is nondecreasing and $J_\beta$ is nonincreasing. Comparing their values at $s=0$ and $s=t$ gives \eqref{eq:lower-comparison} and \eqref{eq:upper-comparison}.
\end{proof}

\subsection{Li--Yau and Harnack estimates}\label{subsec:li-yau}

M\"unch's finite-graph maximum-principle calculation gives the modified Li--Yau estimate
under $\mathrm{CD}(0,n)$ \cite[Theorem~3.1]{MunchCDN}; Pajot and Russ adapted that argument
to infinite bounded-geometry graphs \cite[Section~7.1]{RussPajot}. The proof below follows this
scheme, with the global finite-dimensional curvature term replaced by the extremal point-mass
estimate from Proposition~\ref{prop:extremal-CD} and with a distance penalisation that makes
the $\ell^\infty$ argument valid without a volume-growth assumption.

The extremal estimate in Proposition~\ref{prop:extremal-CD}, rather than a global finite-dimensional condition, is enough to prove the modified Li--Yau inequality.

\begin{theorem}[Modified Li--Yau inequality]
\label{thm:modified-li-yau}
Let $u_t$ be the solution in Proposition~\ref{prop:global-flow}. Then
\begin{equation}
  -\Delta u_t(x)\le\frac{N}{2t}
  \qquad(x\in V,\ t>0).
  \label{eq:modified-li-yau}
\end{equation}
\end{theorem}

\begin{proof}
We give a proof that covers infinite graphs without any a priori growth assumption. Fix a root $o\in V$, put
\[
  \rho(x):=\dist(o,x),
\]
and fix $T>0$. The boundedness of $u_t$ on $[0,T]$, together with boundedness of $\Delta$ on
$\ell^\infty(V)$, implies that
\[
  \Psi(t,x):=t\Delta u_t(x)
\]
is bounded on $[0,T]\times V$. Moreover $u\in C^1$ and
$\partial_tu=\Delta u+\Gamma(u)$; the boundedness and bilinearity recorded above imply that
the vector field $\mathcal F$ is $C^1$, hence $u\in C^2$ in the
$\ell^\infty$ norm. Thus $t\mapsto t\Delta u_t$ is continuously differentiable in
$\ell^\infty(V)$, which justifies the time derivative used at the penalised minimum. For $\varepsilon>0$, consider
\begin{equation}
  \Psi_\varepsilon(t,x):=\Psi(t,x)+\varepsilon\rho(x).
  \label{eq:penalised-Psi}
\end{equation}
This function attains a global minimum on $[0,T]\times V$. The assertion is immediate when $V$ is finite. When $V$ is infinite, put
\[
  M_T:=\sup_{[0,T]\times V}\abs{\Psi}<\infty.
\]
We have $\Psi_\varepsilon(0,o)=0$, whereas
\[
  \Psi_\varepsilon(t,x)\ge-M_T+\varepsilon\rho(x)>0
\]
if $\rho(x)>M_T/\varepsilon$. Consequently, the infimum over $[0,T]\times V$ equals the minimum over $[0,T]$ times a finite ball and is therefore attained.

Suppose first that the minimum is negative, and let $(t_\varepsilon,x_\varepsilon)$ be a minimiser. Then $t_\varepsilon>0$ and
\begin{equation}
  m:=\Delta u_{t_\varepsilon}(x_\varepsilon)<0.
  \label{eq:m-definition}
\end{equation}
If $t_\varepsilon<T$, the time derivative of $t\Delta u_t(x_\varepsilon)$ vanishes. If $t_\varepsilon=T$, its left derivative is nonpositive; because the solution is $C^1$ beyond $T$, this left derivative equals the ordinary derivative at $T$. Thus in either case
\begin{equation}
  \left.\partial_t^-\bigl(t\Delta u_t(x_\varepsilon)\bigr)
  \right|_{t=t_\varepsilon}\le0.
  \label{eq:time-minimum}
\end{equation}
Spatial minimality and the fact that $\rho$ is $1$-Lipschitz give, for every $y\sim x_\varepsilon$,
\begin{equation}
  \Delta u_{t_\varepsilon}(y)-m
  \ge-\frac{\varepsilon}{t_\varepsilon}.
  \label{eq:approx-spatial-minimum}
\end{equation}
Applying \eqref{eq:approx-extremal-CD} with $\eta=\varepsilon/t_\varepsilon$, and then using $d_{x_\varepsilon}\le d_*$ and $d_{x_\varepsilon}(d_{x_\varepsilon}+3)\le N$, gives
\begin{equation}
  \Gamma_2(u_{t_\varepsilon})(x_\varepsilon)
  \ge\frac1N
    \left(-m-\frac{d_*\varepsilon}{2t_\varepsilon}\right)_+^2.
  \label{eq:Gamma2-at-min}
\end{equation}

Using \eqref{eq:bochner},
\begin{align}
  \partial_t(t\Delta u_t)
  &=\Delta u_t+t\bigl(\Delta^2u_t+\Delta\Gamma(u_t)\bigr)\notag\\
  &=\Delta u_t+t\bigl(
      \Delta^2u_t+2\Gamma(u_t,\Delta u_t)+2\Gamma_2(u_t)
    \bigr).
  \label{eq:time-evolution-LY}
\end{align}
In the next two displays, all functions are evaluated at time $t_\varepsilon$. At $x_\varepsilon$ the transport part is
\begin{equation}
  \Delta^2u+2\Gamma(u,\Delta u)
  =\sum_{y\sim x_\varepsilon}
    \bigl(\Delta u(y)-m\bigr)
    \bigl(1+u(y)-u(x_\varepsilon)\bigr).
  \label{eq:transport-identity}
\end{equation}
Every second factor belongs to $[0,2]$ by \eqref{eq:edge-smallness}. Hence \eqref{eq:approx-spatial-minimum} implies
\begin{equation}
  \Delta^2u+2\Gamma(u,\Delta u)
  \ge-\frac{2d_*\varepsilon}{t_\varepsilon}.
  \label{eq:transport-lower-bound}
\end{equation}
Combining \eqref{eq:time-minimum}, \eqref{eq:Gamma2-at-min}, \eqref{eq:time-evolution-LY}, and \eqref{eq:transport-lower-bound}, and setting
\[
  z:=-t_\varepsilon m>0,
\]
we obtain
\begin{equation}
  0\ge-z-2d_*\varepsilon t_\varepsilon
  +\frac2N\left(z-\frac{d_*\varepsilon}{2}\right)_+^2.
  \label{eq:quadratic-z}
\end{equation}
Let
\[
  a_\varepsilon:=\frac{d_*\varepsilon}{2},
  \qquad
  b_\varepsilon:=2d_*\varepsilon T.
\]
If $z\le a_\varepsilon$, the bound below is automatic. Otherwise $w:=z-a_\varepsilon>0$ satisfies
\[
  \frac{2w^2}{N}\le w+a_\varepsilon+b_\varepsilon,
\]
so
\begin{equation}
  z\le a_\varepsilon
  +\frac N4\left(
     1+\sqrt{1+\frac{8(a_\varepsilon+b_\varepsilon)}N}
   \right)
  =\frac N2+o_{\varepsilon\downarrow0}(1).
  \label{eq:z-bound}
\end{equation}
At the minimiser,
\[
  \Psi_\varepsilon(t_\varepsilon,x_\varepsilon)
  =-z+\varepsilon\rho(x_\varepsilon)\ge-z.
\]
Therefore, for every fixed $(t,x)\in[0,T]\times V$,
\begin{align}
  t\Delta u_t(x)
  \ge{}&-a_\varepsilon
  -\frac N4\left(
     1+\sqrt{1+\frac{8(a_\varepsilon+b_\varepsilon)}N}
   \right)
  -\varepsilon\rho(x).
  \label{eq:prelimit-LY}
\end{align}
If the minimum of $\Psi_\varepsilon$ is nonnegative instead, then $\Psi_\varepsilon(t,x)\ge0$ for every $(t,x)$, so \eqref{eq:prelimit-LY} is immediate. Thus \eqref{eq:prelimit-LY} holds for every $\varepsilon>0$, regardless of the sign of the penalised minimum. Letting $\varepsilon\downarrow0$ yields
\[
  t\Delta u_t(x)\ge-\frac N2.
\]
Since $T$ was arbitrary, \eqref{eq:modified-li-yau} follows.
\end{proof}

The passage from the differential estimate to Harnack control uses the graph path-and-time
splitting argument of Bauer et al.\ \cite[Section~5]{BHLLMY} and, in the modified-flow setting,
the implementations of M\"unch \cite[Theorem~3.2]{MunchCDN} and Pajot and Russ
\cite[Section~7.2]{RussPajot}. The one-dimensional estimate below is the minimal-integral
estimate of \cite[Lemma~5.3]{BHLLMY}, in the form recorded in
\cite[Lemma~5.3]{MunchNonlinear}. We include a direct proof to fix the constants and to make
explicit that only continuity of the scalar function is required.

\begin{lemma}[One-dimensional optimisation]
\label{lem:one-dimensional}
Let $a:[s_1,s_2]\to[0,\infty)$ be continuous. Then
\begin{equation}
  \inf_{s\in[s_1,s_2]}
  \left\{
    \sqrt{2a(s)}-\int_s^{s_2}a(t)\,\dd t
  \right\}
  \le\frac{2}{s_2-s_1}.
  \label{eq:one-dimensional}
\end{equation}
\end{lemma}

\begin{proof}
Put
\[
  \tau:=s_2-s_1,
  \qquad
  C:=\frac2\tau,
  \qquad
  \mathcal A(s):=\int_s^{s_2}a(t)\,\dd t.
\]
Assume for contradiction that
\[
  \sqrt{2a(s)}-\mathcal A(s)>C
  \qquad\text{for every }s\in[s_1,s_2].
\]
Then
\[
  a(s)>\frac12(C+\mathcal A(s))^2.
\]
Since $\mathcal A'=-a$, this implies
\[
  \left(\frac1{C+\mathcal A(s)}\right)'
  =\frac{a(s)}{(C+\mathcal A(s))^2}
  >\frac12.
\]
Integration from $s_1$ to $s_2$ gives
\[
  \frac1C-\frac1{C+\mathcal A(s_1)}
  >\frac\tau2=\frac1C,
\]
which is impossible because the left-hand side is strictly smaller than $1/C$.
\end{proof}

\begin{corollary}[Modified-flow parabolic Harnack inequality]
\label{cor:harnack}
For every $x,y\in V$ and $0<T_1<T_2$,
\begin{equation}
  u_{T_1}(x)-u_{T_2}(y)
  \le\frac N2\log\frac{T_2}{T_1}
  +\frac{2\dist(x,y)^2}{T_2-T_1}.
  \label{eq:harnack}
\end{equation}
\end{corollary}

\begin{proof}
The Li--Yau estimate and \eqref{eq:modified-heat} give
\begin{equation}
  \partial_tu_t\ge\Gamma(u_t)-\frac{N}{2t}.
  \label{eq:differential-harnack}
\end{equation}
First suppose $x\sim y$. For $s\in[T_1,T_2]$, integrate \eqref{eq:differential-harnack} at $x$ on $[T_1,s]$ and at $y$ on $[s,T_2]$. Adding the two inequalities and dropping the nonpositive term $-\int_{T_1}^s\Gamma(u_t)(x)\,\dd t$ gives
\begin{align*}
  u_{T_1}(x)-u_{T_2}(y)
  &\le\frac N2\log\frac{T_2}{T_1}
  +\abs{u_s(x)-u_s(y)}
  -\int_s^{T_2}\Gamma(u_t)(y)\,\dd t\\
  &\le\frac N2\log\frac{T_2}{T_1}
  +\sqrt{2\Gamma(u_s)(y)}
  -\int_s^{T_2}\Gamma(u_t)(y)\,\dd t.
\end{align*}
The second inequality follows from the single-edge contribution to \eqref{eq:edge-gamma}. The function $s\mapsto\Gamma(u_s)(y)$ is continuous because $u$ is norm-continuous and $\Gamma$ is a continuous quadratic map. Minimising over $s$ and applying Lemma~\ref{lem:one-dimensional} gives
\begin{equation}
  u_{T_1}(x)-u_{T_2}(y)
  \le\frac N2\log\frac{T_2}{T_1}
  +\frac2{T_2-T_1}
  \qquad(x\sim y).
  \label{eq:adjacent-harnack}
\end{equation}
For general $x\ne y$, let
\[
  x=x_0,x_1,\ldots,x_\ell=y
\]
be a geodesic, where $\ell=\dist(x,y)$. Divide $[T_1,T_2]$ into $\ell$ equal subintervals and apply \eqref{eq:adjacent-harnack} successively. The logarithmic terms telescope to $(N/2)\log(T_2/T_1)$, while the spatial terms sum to
\[
  \ell\frac2{(T_2-T_1)/\ell}
  =\frac{2\ell^2}{T_2-T_1}.
\]
The case $x=y$ follows directly by integrating \eqref{eq:differential-harnack} and dropping the nonnegative $\Gamma$ term.
\end{proof}

\section{Volume doubling and polynomial growth}\label{sec:doubling-proof}

We now combine the semigroup smoothing from Subsection~\ref{subsec:resolvent}, the nonlinear
comparisons from Subsection~\ref{subsec:modified-flow}, and the modified-flow Harnack estimate
from Subsection~\ref{subsec:li-yau}. The cutoff--Harnack--mass-conservation architecture is due
to M\"unch \cite[Section~4.1]{MunchCDN}. Pajot and Russ adapted it to infinite volume by
decomposing the exponential cutoff into a constant background and an $\ell^1$ perturbation,
and by applying mass conservation only to the latter \cite[Section~8]{RussPajot}. We use
precisely this infinite-volume device, while replacing the finite-dimensional curvature inputs by
the point-mass and resolvent estimates proved above.

Fix the parameters from \eqref{eq:N-definition} and \eqref{eq:theta-beta}. Put
\begin{equation}
  \ell_*:=\log(64Nd_*),
  \qquad
  C:=16N\ell_*.
  \label{eq:C-definition}
\end{equation}
Define
\begin{equation}
  Q_0:=2\log2
  +\frac N2\log(32d_*C^2)
  +\frac N2.
  \label{eq:Q0-definition}
\end{equation}
The deliberately generous choice of $C$ ensures a strict gap.

\begin{lemma}[Parameter gap]
\label{lem:parameter-gap}
With \eqref{eq:C-definition}--\eqref{eq:Q0-definition},
\begin{equation}
  C-Q_0\ge1.
  \label{eq:parameter-gap}
\end{equation}
\end{lemma}

\begin{proof}
Since $d_*\ge2$, one has $N\ge10$ and $\ell_*>7$. Moreover,
\[
  \log(16N)\le\ell_*,
  \qquad
  \log\ell_*\le\frac{\ell_*}{2},
  \qquad
  \log(32d_*)\le\ell_*.
\]
Thus
\[
  \log C\le\frac{3\ell_*}{2}
\]
and
\[
  \log(32d_*C^2)\le4\ell_*.
\]
It follows that
\[
  Q_0\le2\log2+2N\ell_*+\frac N2\le3N\ell_*,
\]
because $2\log2+N/2\le N\ell_*$. Consequently,
\[
  C-Q_0\ge13N\ell_*>1.
\]
\end{proof}

Let
\begin{equation}
  r_0:=\left\lceil 2C\sqrt{d_*}\right\rceil.
  \label{eq:r0-definition}
\end{equation}

\begin{proof}[Proof of Theorem~\ref{thm:doubling}]
The cases $d_*\le1$ were settled at the beginning of Section~\ref{sec:preliminary}, so assume $d_*\ge2$. Fix $x\in V$ and an integer $r\ge r_0$. Set
\begin{equation}
  u_0(y):=-C\min\left\{\frac{\dist(x,y)}r,1\right\}.
  \label{eq:cutoff-u0}
\end{equation}
For every $z\in V$, the edge Lipschitz bound gives
\begin{equation}
  \Gamma(u_0)(z)
  \le\frac12\sum_{y\sim z}\frac{C^2}{r^2}
  \le\frac{d_*C^2}{2r^2}
  \le\frac18
  <\frac12.
  \label{eq:cutoff-gradient}
\end{equation}
Let $u_t$ be the global solution of \eqref{eq:modified-heat} with this initial datum.

The first step is a lower bound at the centre for a suitable short time. By \eqref{eq:semigroup-shift} and \eqref{eq:cutoff-gradient},
\begin{equation}
  \norm{P_tu_0-u_0}_\infty
  \le\frac{C\sqrt{Nd_*t}}r.
  \label{eq:u0-smoothing}
\end{equation}
Since $u_0(x)=0$, the lower comparison \eqref{eq:lower-comparison} and $\e^z\ge1+z$ yield
\begin{align}
  \e^{\theta u_t(x)}
  &\ge P_t\e^{\theta u_0}(x)
   \ge1+\theta P_tu_0(x)\notag\\
  &\ge1-\frac{\theta C\sqrt{Nd_*t}}r.
  \label{eq:centre-lower-pre}
\end{align}
Choose
\begin{equation}
  t_0:=\frac{r^2}{4\theta^2C^2Nd_*}
  =\frac{r^2}{C^2Nd_*}.
  \label{eq:t0-definition}
\end{equation}
Then \eqref{eq:centre-lower-pre} gives
\begin{equation}
  u_{t_0}(x)\ge-\frac{\log2}{\theta}=-2\log2.
  \label{eq:centre-lower}
\end{equation}

Put
\begin{equation}
  R_0:=2r,
  \qquad
  T:=\frac{8R_0^2}{N}.
  \label{eq:R0-T-definition}
\end{equation}
A direct calculation gives
\begin{equation}
  \frac{T}{t_0}=32d_*C^2>2.
  \label{eq:T-t0-ratio}
\end{equation}
For $y\in B(x,R_0)$, the Harnack estimate \eqref{eq:harnack}, \eqref{eq:centre-lower}, and $T-t_0\ge T/2$ imply
\begin{align}
  u_T(y)
  &\ge u_{t_0}(x)
  -\frac N2\log\frac{T}{t_0}
  -\frac{2R_0^2}{T-t_0}\notag\\
  &\ge-2\log2
  -\frac N2\log(32d_*C^2)
  -\frac N2
  =-Q_0.
  \label{eq:ball-lower}
\end{align}

We next use the large-exponent comparison. Let
\[
  V_x(s):=\#B(x,s).
\]
Summing \eqref{eq:upper-comparison} over $B(x,R_0)$ and using \eqref{eq:ball-lower},
\begin{equation}
  \e^{-\beta Q_0}V_x(R_0)
  \le\sum_{y\in B(x,R_0)}P_T\e^{\beta u_0}(y).
  \label{eq:mass-start}
\end{equation}
Write
\begin{equation}
  \e^{\beta u_0}=\e^{-\beta C}+h.
  \label{eq:baseline-subtraction}
\end{equation}
Then $0\le h\le1$ and $\supp h\subseteq B(x,r)$, so in particular $h\in\ell^1(V)$. The background subtraction in \eqref{eq:baseline-subtraction} is essential on an infinite graph: mass conservation is applied only to the integrable perturbation $h$, never to the constant baseline. By \eqref{eq:mass-conservation},
\begin{align}
  \e^{-\beta Q_0}V_x(R_0)
  &\le\e^{-\beta C}V_x(R_0)
  +\sum_{y\in B(x,R_0)}P_Th(y)\notag\\
  &\le\e^{-\beta C}V_x(R_0)
  +\sum_{y\in V}P_Th(y)\notag\\
  &=\e^{-\beta C}V_x(R_0)+\sum_{y\in V}h(y)\notag\\
  &\le\e^{-\beta C}V_x(R_0)+V_x(r).
  \label{eq:mass-conservation-use}
\end{align}
Using $C-Q_0\ge1$, we obtain
\begin{equation}
  V_x(2r)=V_x(R_0)
  \le K_{\mathrm{heat}}V_x(r),
  \qquad
  K_{\mathrm{heat}}:=\frac{\e^{\beta Q_0}}{1-\e^{-\beta}}.
  \label{eq:heat-doubling}
\end{equation}
This holds uniformly for every integer $r\ge r_0$.

For $1\le r<r_0$, the number of vertices at distance $k$ from $x$ is at most $d_*^k$. Since $B(x,2r)$ contains only vertices at integer distance less than $2r$, the deliberately coarse degree estimate gives
\begin{equation}
  V_x(2r)
  \le\sum_{k=0}^{2r_0}d_*^k
  \le d_*^{2r_0+1}
  =:K_{\mathrm{loc}},
  \label{eq:local-doubling}
\end{equation}
while $V_x(r)\ge1$. Combining the elementary and nontrivial degree cases, a valid global doubling constant is
\begin{equation}
  K(d_*):=
  \begin{cases}
    2,&d_*\le1,\\[1mm]
    \displaystyle
    \max\left\{
      d_*^{2r_0+1},
      \frac{\e^{2Q_0}}{1-\e^{-2}}
    \right\},&d_*\ge2,
  \end{cases}
  \label{eq:K-final}
\end{equation}
where, in the second branch, $N,\ell_*,C,Q_0,r_0$ are given by \eqref{eq:N-definition}, \eqref{eq:C-definition}, \eqref{eq:Q0-definition}, and \eqref{eq:r0-definition}. This proves Theorem~\ref{thm:doubling}.
\end{proof}

The corollary is the direct dyadic iteration of the theorem.

\begin{proof}
For Corollary~\ref{cor:poly}, because the balls are open,
\[
B(x,1)=\{x\}.
\]
Iterating Theorem~\ref{thm:doubling} gives
\[
\#B(x,2^k)\leq K(d_*)^k.
\]
Let $r\geq2$ be an integer and put
\[
k:=\lceil\log_2r\rceil.
\]
Then
\begin{align*}
\#B(x,r)
&\leq K(d_*)^k\\
&\leq K(d_*)r^{\log_2K(d_*)}\\
&\leq r^{2\log_2K(d_*)}.
\end{align*}
The last inequality uses $r\geq2$ and $K(d_*)\geq2$. The case $r=1$ is an equality. Taking
the ceiling of the exponent gives the stated value of $D(d_*)$.
\end{proof}

\section{Heat-kernel localisation for Poincar\'e inequalities}\label{sec:poincare-localisation}

This section proves the local analytic estimate needed for Theorem~\ref{thm:poincare-two}.
First, reverse Poincar\'e smoothing controls the total-variation distance between nearby
heat-kernel rows. The semigroup estimate from Subsection~\ref{subsec:resolvent} gives a
uniform first-moment displacement bound for the associated walk, and a standard
strong-Markov maximal argument converts it into an exit-time estimate. We then compare the
full walk with the finite-state chain obtained by rejecting jumps that leave a large ball. A
finite-dimensional localisation lemma converts the resulting row overlap into a spectral-gap
estimate. Spectral calculus yields a scale-invariant Poincar\'e inequality with a fixed
degree-dependent dilation. Volume doubling is not used in this section.

\subsection{Heat-kernel row contraction and exit-time control}\label{subsec:exit-control}

For probability measures on a countable set, use the convention
\[
\|\mu-\nu\|_{\TV}
:=\sup_{0\leq h\leq1}|\mu(h)-\nu(h)|.
\]

The reverse Poincar\'e estimate below is the standard graph-semigroup consequence of
$\mathrm{CD}(0,\infty)$; see, for example, \cite{LinLiu}. We include the short proof to fix the
normalisation and the total-variation constant.

\begin{lemma}[Reverse Poincar\'e and heat-kernel row contraction]\label{lem:reverse-poincare}
For every bounded $h:V\to\R$ and every $t>0$,
\begin{equation}\label{eq:reverse-poincare}
P_t(h^2)-(P_th)^2
\geq 2t\,\Gamma(P_th)
\end{equation}
pointwise. Consequently,
\begin{equation}\label{eq:tv-contraction}
\|p_t(x,\cdot)-p_t(y,\cdot)\|_{\TV}
\leq
\frac{\dist(x,y)}{2\sqrt t}
\end{equation}
for all $x,y\in V$.
\end{lemma}

\begin{proof}
Fix $t>0$. Since $\Delta$ is bounded, $s\mapsto P_s$ is differentiable in operator
norm. Multiplication is a bounded bilinear map on $\ell^\infty(V)$, so
$s\mapsto(P_{t-s}h)^2$ is continuously differentiable in $\ell^\infty(V)$. Therefore
\[
\frac{\dd}{\dd s}
P_s\bigl((P_{t-s}h)^2\bigr)
=2P_s\Gamma(P_{t-s}h)
\]
in the $\ell^\infty$ norm, and the ensuing integrals are Bochner integrals.
After integration,
\[
P_t(h^2)-(P_th)^2
=2\int_0^tP_s\Gamma(P_{t-s}h)\,\dd s.
\]
For a fixed $s\in[0,t]$, apply the interpolation identity
\eqref{eq:interpolation} over the time interval $s$ to the function $P_{t-s}h$. This gives
\[
P_s\Gamma(P_{t-s}h)-\Gamma(P_th)
=2\int_0^sP_u\Gamma_2(P_{t-u}h)\,\dd u
\geq0.
\]
Substitution into the preceding integral proves \eqref{eq:reverse-poincare}.
Now let $0\leq h\leq1$. Since
\[
P_t(h^2)-(P_th)^2
\leq P_th-(P_th)^2
\leq\frac14,
\]
we obtain
\[
\Gamma(P_th)\leq\frac1{8t}.
\]
If $x\sim y$, the single-edge contribution at $x$ gives
\[
\frac12\bigl(P_th(y)-P_th(x)\bigr)^2
\leq\Gamma(P_th)(x).
\]
Therefore
\[
|P_th(y)-P_th(x)|
\leq\frac1{2\sqrt t}.
\]
Summing along a geodesic from $x$ to $y$ and taking the supremum over $0\leq h\leq1$ proves
\eqref{eq:tv-contraction}.
\end{proof}

Let $(X_t)_{t\geq0}$ be the continuous-time Markov chain with semigroup $P_t$. It can be
realised by rate-$d_*$ Poissonisation of the uniformising kernel $I+\Delta/d_*$. The Poisson
clock has finitely many rings on every bounded time interval, so the chain is nonexplosive;
the standard construction also gives the strong Markov property. See, for example,
\cite{LevinPeres}.

The exit estimate below is an Ottaviani--Skorokhod-type strong-Markov maximal argument
applied to the first-moment bound furnished by \eqref{eq:semigroup-shift}; compare
\cite{KuhnSchillingMaximal}. We include the argument because only the weak diffusive bound of
order $\sqrt t/R$ is needed.

\begin{lemma}[Diffusive displacement and exit control]\label{lem:exit-control}
For every $x\in V$ and every $t\geq0$,
\begin{equation}\label{eq:displacement}
\mathbb E_x\dist(x,X_t)\leq\sqrt{Nd_*t}.
\end{equation}
For $R>0$, let
\[
\tau_R:=\inf\{s\geq0:\dist(x,X_s)\geq R\}.
\]
If $R\geq4\sqrt{Nd_*t}$, then
\begin{equation}\label{eq:exit-control}
\mathbb P_x(\tau_R\leq t)
\leq\frac{4\sqrt{Nd_*t}}R.
\end{equation}
\end{lemma}

\begin{proof}
For $M>0$, put
\[
\rho_M(z):=\min\{\dist(x,z),M\}.
\]
This function is bounded and $1$-Lipschitz on edges. Hence
\[
\Gamma(\rho_M)(z)
\leq\frac12\sum_{w\sim z}1
\leq\frac{d_*}{2}.
\]
Since $\rho_M(x)=0$, the semigroup estimate \eqref{eq:semigroup-shift} gives
\[
P_t\rho_M(x)
\leq\sqrt{Nd_*t}.
\]
Letting $M\to\infty$ and using monotone convergence proves \eqref{eq:displacement}.

On the event $\{\tau_R\leq t\}$, set $Z:=X_{\tau_R}$. If also
$\dist(x,X_t)<R/2$, then the triangle inequality gives $\dist(Z,X_t)\geq R/2$. Therefore,
on $\{\tau_R\leq t\}$, the strong Markov property, \eqref{eq:displacement}, and Markov's
inequality give
\begin{align*}
&\mathbb P_x\left(
  \dist(x,X_t)<\frac R2
  \,\middle|\,
  \mathcal F_{\tau_R}
\right)\\
&\qquad\leq
\sup_{\substack{z\in V\\0\leq s\leq t}}
\mathbb P_z\left(
  \dist(z,X_s)\geq\frac R2
\right)\\
&\qquad\leq
\frac{2\sqrt{Nd_*t}}R
\leq\frac12.
\end{align*}
Consequently,
\[
\mathbb P_x\left(
  \tau_R\leq t,\ \dist(x,X_t)<\frac R2
\right)
\leq\frac12\mathbb P_x(\tau_R\leq t).
\]
Decomposing according to the endpoint position,
\begin{align*}
\mathbb P_x(\tau_R\leq t)
&\leq
\mathbb P_x\left(\dist(x,X_t)\geq\frac R2\right)
+\frac12\mathbb P_x(\tau_R\leq t).
\end{align*}
A final application of Markov's inequality and \eqref{eq:displacement} yields
\[
\mathbb P_x(\tau_R\leq t)
\leq2\mathbb P_x\left(\dist(x,X_t)\geq\frac R2\right)
\leq\frac{4\sqrt{Nd_*t}}R.
\]
\end{proof}

\subsection{Censored kernels and localisation}\label{subsec:censoring}

The full heat kernel is not supported in a finite ball, whereas the desired energy is local.
We therefore compare the full walk with the finite-state walk obtained by rejecting every
attempted jump from a large ball to its complement. The exit estimate from
Subsection~\ref{subsec:exit-control} makes the coupling error small, and the resulting row
overlap is converted into a spectral-gap estimate by a finite-dimensional localisation lemma.

The broad finite-volume semigroup strategy has classical precedents. A
Kusuoka--Stroock-type argument converts lower bounds for a localised heat kernel into a
Poincar\'e estimate through finite-time semigroup dissipation; see
\cite[Theorem~5.10]{KusuokaStroockIII} and the modern formulation
\cite[Lemma~5.5]{GISCPoincare}. In the graph setting, Delmotte
\cite[Section~3.4, especially Theorem~3.11]{Delmotte} restricts the walk to an enlarged ball, obtains a lower bound
for the resulting finite-volume kernel on the smaller ball, and controls the semigroup
dissipation by the local edge energy. Our proof follows this general architecture, but replaces
a pointwise lower heat-kernel bound by total-variation overlap of rows and the
posterior/Dobrushin lemma below.

For a finite set $D\subset V$, define the censored, or induced-subgraph, generator
\begin{equation}\label{eq:censored-generator}
\Delta_Dg(x)
:=\sum_{\substack{y\in D\\y\sim x}}
\bigl(g(y)-g(x)\bigr),
\qquad x\in D.
\end{equation}
Write
\[
P_t^D:=e^{t\Delta_D}
\]
and let $p_t^D(x,y)$ be its kernel. Under rate-$d_*$ uniformisation, a proposed jump from
$D$ to $D^c$ is rejected and becomes an additional holding event. The kernel is symmetric and
stochastic with respect to counting measure on $D$. When it is compared with a probability
measure on $V$, it is extended by zero on $D^c$.

\begin{lemma}[Censored finite-volume row overlap]\label{lem:censored-overlap}
Set
\begin{equation}\label{eq:Lambda}
\Lambda
:=1+64\sqrt{Nd_*}
=1+64d_*\sqrt{d_*+3}.
\end{equation}
Fix $o\in V$ and an integer $r\geq1$. Put
\[
B:=B(o,r),
\qquad
D:=B(o,\Lambda r),
\qquad
t:=4r^2.
\]
The set $D$ is finite by Assumption~\ref{ass:degree}, and
\[
\sup_{x,y\in B}
\|p_t^D(x,\cdot)-p_t^D(y,\cdot)\|_{\TV}
\leq\frac34.
\]
\end{lemma}

\begin{proof}
For $x,y\in B$,
\[
\dist(x,y)<2r.
\]
Therefore Lemma~\ref{lem:reverse-poincare} gives
\[
\|p_t(x,\cdot)-p_t(y,\cdot)\|_{\TV}
\leq\frac12.
\]
Fix $x\in B$. If $D=V$, there is nothing to compare. Otherwise, every point of $D^c$ is at
distance greater than $(\Lambda-1)r$ from $x$. Since
\[
\sqrt{Nd_*t}=2r\sqrt{Nd_*}
\]
and
\[
(\Lambda-1)r=32\sqrt{Nd_*t},
\]
Lemma~\ref{lem:exit-control} shows that the original walk started at $x$ leaves $D$ by time
$t$ with probability at most $1/8$.

For completeness, couple the original and censored walks with the same rate-$d_*$ Poisson
clock. Whenever their common position is $z\in D$, sample a transition according to the
uniformising kernel $I+\Delta/d_*$. If the sampled point lies in $D$, move both chains there.
If it lies in $D^c$, move only the original chain and keep the censored chain at $z$. After the
first disagreement, continue with any coupling. The censored uniformising kernel has holding
probability
\[
1-\frac{\deg_D(z)}{d_*},
\]
where $\deg_D(z)$ is the number of neighbours of $z$ in $D$. Running this kernel at rate
$d_*$ gives the generator \eqref{eq:censored-generator}. The chains agree until the original
walk leaves $D$. By the coupling characterisation of total variation,
\[
\|p_t^D(x,\cdot)-p_t(x,\cdot)\|_{\TV}
\leq \mathbb P_x(X_s\notin D\text{ for some }s\leq t)
\leq\frac18.
\]
Applying the same estimate to $y$ and using the triangle inequality,
\[
\|p_t^D(x,\cdot)-p_t^D(y,\cdot)\|_{\TV}
\leq\frac18+\frac12+\frac18
=\frac34.
\]
\end{proof}

The next lemma is finite-dimensional and does not use curvature. Its Dobrushin step is
standard, while the posterior-kernel construction packages row overlap in a form adapted to
the variance on the smaller set. It may be viewed as a total-variation analogue of the
pointwise lower-bound step in the preceding heat-kernel arguments.

\begin{lemma}[Kernel localisation from row overlap]\label{lem:kernel-localisation}
Let $D$ be finite, let $B\subset D$ be nonempty, and let $\mathsf K$ be a symmetric Markov
kernel on $D$. Suppose
\[
\sup_{x,y\in B}
\|\mathsf K(x,\cdot)-\mathsf K(y,\cdot)\|_{\TV}
\leq\kappa<1.
\]
Then every $f:D\to\R$ satisfies
\begin{equation}\label{eq:kernel-localisation}
\sum_{x\in B}|f(x)-f_B|^2
\leq
\frac{2}{1-\kappa}
\ip{f}{(I-\mathsf K)f}_{\ell^2(D)}.
\end{equation}
\end{lemma}

\begin{proof}
For $z\in D$, define
\[
s(z):=\sum_{a\in B}\mathsf K(a,z).
\]
Define a kernel $\mathsf Q$ on $B$ by
\[
\mathsf Q(x,y)
:=\sum_{\substack{z\in D\\s(z)>0}}
\frac{\mathsf K(x,z)\mathsf K(y,z)}{s(z)}.
\]
The kernel $\mathsf Q$ is symmetric. It is stochastic because, for $x\in B$,
\begin{align*}
\sum_{y\in B}\mathsf Q(x,y)
&=\sum_{\substack{z\in D\\s(z)>0}}
\mathsf K(x,z)\frac{\sum_{y\in B}\mathsf K(y,z)}{s(z)}\\
&=\sum_{\substack{z\in D\\s(z)>0}}\mathsf K(x,z)\\
&=1.
\end{align*}
The last equality holds because $\mathsf K(x,z)>0$ with $x\in B$ implies $s(z)>0$.
For $s(z)>0$, define
\[
\Pi(z,y):=\frac{\mathsf K(y,z)}{s(z)},
\qquad y\in B,
\]
and choose an arbitrary probability row when $s(z)=0$. Regard $\mathsf K$ as a kernel from
$B$ to $D$. Then
\[
\mathsf Q=\mathsf K\Pi.
\]
Total variation contracts under a Markov kernel, so
\[
\sup_{x,y\in B}
\|\mathsf Q(x,\cdot)-\mathsf Q(y,\cdot)\|_{\TV}
\leq\kappa.
\]
Consequently,
\[
\osc(\mathsf Qh)\leq\kappa\,\osc(h)
\]
for every $h:B\to\R$. This is the standard Dobrushin oscillation contraction; see
\cite{GaubertQu}.

Because $\mathsf Q$ is symmetric and stochastic, it is self-adjoint on the uniform
probability space $\pi_B$. Let
\[
\phi_1\equiv1,\phi_2,\ldots,\phi_{\#B}
\]
be an orthonormal eigenbasis, with
\[
\mathsf Q\phi_j=\lambda_j\phi_j.
\]
If $j\geq2$, then $\phi_j$ is nonconstant and
\[
|\lambda_j|\osc(\phi_j)
=\osc(\mathsf Q\phi_j)
\leq\kappa\osc(\phi_j).
\]
Thus
\[
|\lambda_j|\leq\kappa.
\]
Writing
\[
h-h_B=\sum_{j=2}^{\#B}a_j\phi_j,
\]
we obtain
\begin{align*}
\ip{h}{(I-\mathsf Q)h}_{\ell^2(\pi_B)}
&=\sum_{j=2}^{\#B}(1-\lambda_j)a_j^2\\
&\geq(1-\kappa)\sum_{j=2}^{\#B}a_j^2\\
&=(1-\kappa)\Var_{\pi_B}(h).
\end{align*}
Therefore
\begin{equation}\label{eq:q-poincare}
\Var_{\pi_B}(h)
\leq
\frac1{1-\kappa}
\ip{h}{(I-\mathsf Q)h}_{\ell^2(\pi_B)}.
\end{equation}
It remains to compare the two Dirichlet forms. Let $X$ be uniform on $B$. Conditional on
$X=x$, let $Z$ have law $\mathsf K(x,\cdot)$. Conditional on $Z=z$, sample $Y$
independently from the posterior law $\Pi(z,\cdot)$. Then the transition kernel from $X$ to
$Y$ is $\mathsf Q$, and $X$ and $Y$ are conditionally independent with the same conditional
law given $Z$. Hence
\begin{align*}
\ip{f}{(I-\mathsf Q)f}_{\ell^2(\pi_B)}
&=\frac12\mathbb E\bigl(f(X)-f(Y)\bigr)^2\\
&=\mathbb E\Var\bigl(f(X)\mid Z\bigr)\\
&\leq\mathbb E\bigl(f(X)-f(Z)\bigr)^2.
\end{align*}
The inequality is the $L^2$ optimality of conditional expectation, because $f(Z)$ is a
$Z$-measurable predictor. Expanding the last expectation,
\begin{align*}
\mathbb E\bigl(f(X)-f(Z)\bigr)^2
&=\frac1{\#B}
\sum_{\substack{x\in B\\z\in D}}
\mathsf K(x,z)\bigl(f(x)-f(z)\bigr)^2\\
&\leq\frac1{\#B}
\sum_{x,z\in D}
\mathsf K(x,z)\bigl(f(x)-f(z)\bigr)^2.
\end{align*}
Symmetry and stochasticity of $\mathsf K$ give
\[
\sum_{x,z\in D}
\mathsf K(x,z)\bigl(f(x)-f(z)\bigr)^2
=2\ip{f}{(I-\mathsf K)f}_{\ell^2(D)}.
\]
Apply \eqref{eq:q-poincare} to $h=f|_B$ and multiply by $\#B$. This proves
\eqref{eq:kernel-localisation}.
\end{proof}

\subsection{A fixed-dilation Poincar\'e inequality}\label{subsec:fixed-poincare}

We now apply the overlap estimate from Subsection~\ref{subsec:censoring} to the censored heat
semigroup. Spectral calculus turns one-step dissipation into edge energy and yields the
fixed-dilation inequality.

\begin{theorem}[Fixed-dilation Poincar\'e inequality]\label{thm:poincare-fixed}
Suppose Assumptions~\ref{ass:degree} and~\ref{ass:curvature} hold and $d_*\geq2$. Let $\Lambda$ be given by
\eqref{eq:Lambda}. Then
\[
\sum_{x\in B(o,r)}|f(x)-f_{B(o,r)}|^2
\leq32r^2\mathcal E_{B(o,\Lambda r)}(f)
\]
for every $o\in V$, every integer $r\geq1$, and every $f:V\to\R$.
\end{theorem}

\begin{proof}
Fix $o$ and $r$, and use the notation
\[
B:=B(o,r),
\qquad
D:=B(o,\Lambda r),
\qquad
t:=4r^2.
\]
Apply Lemma~\ref{lem:kernel-localisation} to
\[
\mathsf K:=P_t^D
\]
and
\[
\kappa:=\frac34.
\]
Lemma~\ref{lem:censored-overlap} gives the required total-variation hypothesis. For the
restriction of $f$ to $D$,
\[
\sum_{x\in B}|f(x)-f_B|^2
\leq8\ip{f}{(I-P_t^D)f}_{\ell^2(D)}.
\]
The operator $-\Delta_D$ is nonnegative and self-adjoint. Since
\[
1-e^{-t\lambda}\leq t\lambda
\qquad
(\lambda\geq0),
\]
spectral calculus gives
\[
\ip{f}{(I-P_t^D)f}_{\ell^2(D)}
\leq t\ip{f}{-\Delta_Df}_{\ell^2(D)}.
\]
By symmetry of the graph,
\begin{align*}
\ip{f}{-\Delta_Df}_{\ell^2(D)}
&=\sum_{\substack{\{x,y\}\in E\\x,y\in D}}
\bigl(f(x)-f(y)\bigr)^2\\
&=\mathcal E_D(f).
\end{align*}
Therefore
\[
\sum_{x\in B}|f(x)-f_B|^2
\leq8t\mathcal E_D(f)
=32r^2\mathcal E_D(f).
\]
\end{proof}

\begin{remark}
Theorem~\ref{thm:poincare-fixed} is already scale invariant: its dilation is independent of the
centre and radius. Its proof does not use volume doubling. The doubling theorem is used only
to replace the degree-dependent dilation $\Lambda$ by the conventional dilation two.
\end{remark}

\section{Reduction to dilation two}\label{sec:poincare-proof}

In standard geodesic doubling metric-measure formulations, a weak Poincar\'e inequality
with any fixed enlargement can be upgraded to the strong form by a Whitney-covering
argument going back to Jerison; see \cite{Jerison}, the detailed account in
\cite[Sections~5.3.2--5.3.5]{SaloffCosteSobolev}, and the metric-measure/random-walk formulation in
\cite[Chapter~3 in the arXiv version]{MuruganSaloffCoste}. Our inequality is stated for open integer graph
balls and uses the induced edge energy, and we also need to retain quantitative dependence
only on $d_*$. For completeness, we therefore include a discrete covering-and-chaining
argument adapted to these conventions, yielding the conventional dilation two.

The volume-doubling theorem bounds the number of balls in a separated covering. Local
Poincar\'e inequalities control oscillation inside each ball, and a connected graph of centres
controls the differences of their local means. Small radii are handled directly by geodesics.

\begin{proof}
If $d_*\leq1$, then $G$ has at most two vertices. For $r=1$, the left-hand side is zero. For
$r\geq2$, one has $B(o,r)=B(o,2r)=V$. The one-vertex case is trivial, and on the two-vertex
graph
\[
\sum_{x\in V}|f(x)-f_V|^2
=\frac12\mathcal E_V(f)
\leq r^2\mathcal E_V(f).
\]
Thus $C_P(d_*)=1$ suffices when $d_*\leq1$. Assume henceforth that $d_*\geq2$. Put
\[
C_0:=32
\]
and retain $\Lambda$ from \eqref{eq:Lambda}. Let
\[
K_D:=K(d_*)
\]
be the doubling constant from Theorem~\ref{thm:doubling}. Define
\[
R_*:=\lceil64\Lambda\rceil,
\qquad
q_0:=\left\lceil\log_2(128\Lambda)\right\rceil,
\qquad
M_{\mathrm{cov}}:=K_D^{q_0+1}.
\]
We first treat $r\geq R_*$. Set
\[
s:=\left\lfloor\frac{r}{32\Lambda}\right\rfloor.
\]
Since $r/(32\Lambda)\geq2$,
\begin{equation}\label{eq:s-bounds}
1\leq\frac{r}{64\Lambda}\leq s\leq\frac{r}{32\Lambda}.
\end{equation}
Since $B(o,r)$ is finite, choose a maximal $2s$-separated set
\[
\{x_1,\ldots,x_n\}\subset B(o,r)
\]
containing $x_1=o$. Here $2s$-separated means
\[
\dist(x_i,x_j)\geq2s
\qquad
(i\neq j).
\]
Maximality gives the covering
\begin{equation}\label{eq:cover}
B(o,r)\subset\bigcup_{i=1}^nB(x_i,2s),
\end{equation}
and the balls $B(x_i,s)$ are pairwise disjoint.
Put
\[
q:=\left\lceil\log_2\frac{2r}{s}\right\rceil.
\]
By \eqref{eq:s-bounds},
\[
q\leq q_0.
\]
For every $i$,
\[
B(o,r)\subset B(x_i,2r)\subset B(x_i,2^qs).
\]
Iterated doubling therefore gives
\begin{equation}\label{eq:small-ball-lower}
\#B(x_i,s)
\geq K_D^{-q}\#B(o,r).
\end{equation}
The disjoint balls $B(x_i,s)$ lie in $B(o,2r)$. Hence
\begin{align*}
nK_D^{-q}\#B(o,r)
&\leq\#B(o,2r)\\
&\leq K_D\#B(o,r).
\end{align*}
Thus
\[
n\leq K_D^{q+1}\leq M_{\mathrm{cov}}.
\]
Form a graph on the centres by joining $x_i$ and $x_j$ whenever
\[
\dist(x_i,x_j)<5s.
\]
This centre graph is connected. Indeed, follow a geodesic from $o$ to $x_i$. Assign to each
geodesic vertex a centre whose ball in \eqref{eq:cover} contains it, choosing $x_1=o$ for the
initial vertex and $x_i$ for the final vertex. Centres assigned to two consecutive vertices are
at distance less than
\[
2s+1+2s=4s+1\leq5s.
\]
After deleting repetitions, this gives a path from $x_1$ to $x_i$. Every centre can therefore
be connected to $x_1$ by a simple path with at most $n-1$ links.
Set
\[
B:=B(o,r),
\qquad
\mathcal E:=\mathcal E_{B(o,2r)}(f),
\qquad
U_i:=B(x_i,2s),
\qquad
m_i:=f_{U_i}.
\]
Choose a disjoint partition
\[
B=A_1\mathbin{\dot\cup}\cdots\mathbin{\dot\cup}A_n
\]
with $A_i\subset U_i$. Such a partition exists by \eqref{eq:cover}.
By Theorem~\ref{thm:poincare-fixed}, \eqref{eq:s-bounds}, and
\[
B(x_i,2\Lambda s)\subset B(o,2r),
\]
we have
\begin{equation}\label{eq:local-poincare}
\sum_{x\in U_i}|f(x)-m_i|^2
\leq C_0(2s)^2\mathcal E
=4C_0s^2\mathcal E.
\end{equation}
The stated inclusion follows from $x_i\in B(o,r)$ and
$2\Lambda s\leq r/16$.
Suppose now that $x_i$ and $x_j$ are linked. Put
\[
W:=B(x_i,8s).
\]
Then $U_i\cup U_j\subset W$. Since $B(x_i,s)\subset U_i$ and
$B(x_j,s)\subset U_j$, \eqref{eq:small-ball-lower} gives
\[
\#U_i,\#U_j
\geq K_D^{-q}\#B.
\]
Writing $m_W=f_W$ and using Cauchy--Schwarz,
\begin{align*}
|m_i-m_j|^2
&\leq2|m_i-m_W|^2+2|m_j-m_W|^2\\
&\leq\frac{4K_D^q}{\#B}
\sum_{x\in W}|f(x)-m_W|^2.
\end{align*}
Apply Theorem~\ref{thm:poincare-fixed} at radius $8s$. Since
\[
B(x_i,8\Lambda s)\subset B(o,2r)
\]
because $8\Lambda s\leq r/4$, we obtain
\[
\sum_{x\in W}|f(x)-m_W|^2
\leq C_0(8s)^2\mathcal E.
\]
Therefore
\begin{equation}\label{eq:mean-link}
|m_i-m_j|^2
\leq
\frac{256C_0K_D^qs^2}{\#B}\mathcal E
\end{equation}
whenever the centres are linked.
Join $x_i$ to $x_1$ by a simple centre path. It has at most $n-1$ links. Cauchy--Schwarz and
\eqref{eq:mean-link} yield
\begin{equation}\label{eq:mean-path}
|m_i-m_1|^2
\leq
\frac{256C_0K_D^qn^2s^2}{\#B}\mathcal E.
\end{equation}
Since $f_B$ minimises the squared error over constants,
\begin{align*}
\sum_{x\in B}|f(x)-f_B|^2
&\leq\sum_{x\in B}|f(x)-m_1|^2\\
&\leq2\sum_{i=1}^n\sum_{x\in A_i}|f(x)-m_i|^2
+2\sum_{i=1}^n\#A_i|m_i-m_1|^2.
\end{align*}
Using \eqref{eq:local-poincare}, \eqref{eq:mean-path}, and
$\sum_i\#A_i=\#B$,
\begin{align*}
\sum_{x\in B}|f(x)-f_B|^2
&\leq
\bigl(8C_0n+512C_0K_D^qn^2\bigr)s^2\mathcal E\\
&\leq
C_0\bigl(8M_{\mathrm{cov}}+512K_D^{q_0}M_{\mathrm{cov}}^2\bigr)
 r^2\mathcal E.
\end{align*}
This proves the theorem for $r\geq R_*$.
Now let $1\leq r<R_*$. For each $x\in B(o,r)$, choose a geodesic $\gamma_x$ from $o$ to
$x$. Every vertex of $\gamma_x$ lies in $B(o,r)$, and Cauchy--Schwarz gives
\[
|f(x)-f(o)|^2
\leq
r\sum_{\{u,v\}\in\gamma_x}
\bigl(f(u)-f(v)\bigr)^2
\leq r\mathcal E_{B(o,r)}(f).
\]
Because the mean minimises the squared error,
\[
\sum_{x\in B(o,r)}|f(x)-f_{B(o,r)}|^2
\leq
r\,\#B(o,r)\,\mathcal E_{B(o,r)}(f).
\]
The degree bound gives
\[
\#B(o,r)
\leq
V_*:=\sum_{k=0}^{R_*-1}d_*^k.
\]
Since $r\geq1$ and $B(o,r)\subset B(o,2r)$,
\[
\sum_{x\in B(o,r)}|f(x)-f_{B(o,r)}|^2
\leq
V_*r^2\mathcal E_{B(o,2r)}(f).
\]
Combining this estimate with the case $d_*\leq1$, one valid choice is
\[
C_P(d_*):=
\begin{cases}
1,&d_*\leq1,\\[1mm]
\displaystyle
\max\left\{
V_*,
C_0\bigl(8M_{\mathrm{cov}}+512K_D^{q_0}M_{\mathrm{cov}}^2\bigr)
\right\},&d_*\geq2.
\end{cases}
\]
\end{proof}

\begin{remark}
For every finite $D\subset V$,
\begin{align*}
\mathcal E_D(f)
&\leq\sum_{x\in D}\Gamma(f)(x).
\end{align*}
Thus Theorems~\ref{thm:poincare-fixed} and \ref{thm:poincare-two} also imply the corresponding
formulations with vertex energy on the right-hand side.
\end{remark}

\appendix

\section{Bounded maximum principle on an infinite graph}\label{sec:max-principle}

The nonlinear gradient estimate in Subsection~\ref{subsec:modified-flow} requires a maximum principle on an infinite
vertex set, where the supremum need not be attained. The following bounded version uses
approximate maximisers and an upper Dini derivative; no volume-growth hypothesis is needed.

\begin{lemma}[Bounded maximum principle]\label{lem:max-principle}
Let
\[
h\in C^1\bigl([0,T],\ell^\infty(V)\bigr)
\]
and suppose
\[
\partial_th_t(x)
\leq
\sum_y a_t(x,y)\bigl(h_t(y)-h_t(x)\bigr),
\]
where $a_t(x,y)\geq0$ and
\[
\sup_{t\in[0,T],\,x\in V}\sum_y a_t(x,y)
\leq M<\infty.
\]
Then
\[
t\longmapsto\sup_xh_t(x)
\]
is nonincreasing on $[0,T]$.
\end{lemma}

\begin{proof}
Set
\[
H(t):=\sup_xh_t(x).
\]
For $t<T$, write
\[
D^+H(t)
:=\limsup_{\tau\downarrow0}
\frac{H(t+\tau)-H(t)}\tau.
\]
Norm differentiability of $h$ implies
\begin{equation}\label{eq:dini}
D^+H(t)
\leq
\lim_{\eta\downarrow0}
\sup_{\{x:h_t(x)>H(t)-\eta\}}
\partial_th_t(x).
\end{equation}
For completeness, choose $\tau_n\downarrow0$ realising the limsup in $D^+H(t)$ and choose
$x_n$ so that
\[
h_{t+\tau_n}(x_n)
\geq H(t+\tau_n)-\tau_n^2.
\]
Since $h$ is differentiable in $\ell^\infty$,
\[
h_{t+\tau_n}
=h_t+\tau_n\partial_th_t+o(\tau_n)
\]
in $\ell^\infty$. The same expansion shows that $H$ is locally Lipschitz. Consequently,
\[
h_t(x_n)\geq H(t)-O(\tau_n),
\]
so $x_n$ is an approximate maximiser of $h_t$. Moreover,
\[
\frac{H(t+\tau_n)-H(t)}{\tau_n}
\leq\partial_th_t(x_n)+o(1).
\]
Taking the limsup and then letting the approximation level tend to zero proves
\eqref{eq:dini}.
For every $x$ in the set on the right-hand side of \eqref{eq:dini} and every $y$,
\[
h_t(y)-h_t(x)\leq\eta.
\]
Therefore
\[
\partial_th_t(x)
\leq M\eta.
\]
Letting $\eta\downarrow0$ in \eqref{eq:dini} gives
\[
D^+H(t)\leq0.
\]
Since
\[
|H(t)-H(s)|\leq\|h_t-h_s\|_\infty
\]
and $h$ is $C^1$, the function $H$ is locally Lipschitz and hence absolutely continuous. At
almost every time it is differentiable, and there
\[
H'(t)=D^+H(t)\leq0.
\]
Integration shows that $H$ is nonincreasing. No attainment of the supremum and no volume
hypothesis are required.
\end{proof}

\noindent \textbf{Declarations of interest}: None.

\bigskip
\noindent \textbf{Data availability statement}: There are no new data associated with this article.

\bigskip
\noindent \textbf{AI assistance statement}: The authors used AI tools to assist with initial conceptualization, symbolic checking and manuscript editing; all mathematical validation, final proof decisions, and final wording remain the sole responsibility of the human authors.

\bigskip
\raggedbottom

\end{document}